\documentclass[12pt,oneside]{report}

\usepackage{setspace}

\usepackage{amsthm}

\usepackage{amssymb}   

\usepackage{amsmath}

\usepackage{mathrsfs}

\usepackage{stmaryrd}

\usepackage{enumitem}

\usepackage{MnSymbol,wasysym}

\usepackage[all]{xy}

\usepackage{textcomp}

\usepackage{amsbsy}

\usepackage{datetime}
\renewcommand{\dateseparator}{-}
\renewcommand{\today}{\the\year \dateseparator \twodigit\month
\dateseparator \twodigit\day}

\usepackage
[pdfauthor={Philip Douglas Tynan},
 pdftitle={Equivariant Weiss Calculus and Loop Spaces of Stiefel Manifolds},
 bookmarks=false]
{hyperref}

\title{Equivariant Weiss Calculus and Loop Spaces of Stiefel Manifolds} 
\author{Philip Douglas~Tynan}

\newtheorem{thm}{Theorem}[section]

\newtheorem{thm-defn}[thm]{Theorem/Definition}
\newtheorem{lem}[thm]{Lemma}
\newtheorem{prop}[thm]{Proposition}
\newtheorem{cor}[thm]{Corollary}
\newtheorem{conj}[thm]{Conjecture}

\theoremstyle{definition}
\newtheorem{defn}[thm]{Definition}

\newtheorem{eg}[thm]{Example}

\theoremstyle{remark}
\newtheorem{rmk}[thm]{Remark}

\numberwithin{equation}{section}

\newcommand{\fr}{{\mathfrak r}}

\newcommand{\cC}{{\mathcal C}}

\newcommand{\cE}{{\mathcal E}}

\newcommand{\cG}{{\mathcal G}}

\newcommand{\cJ}{{\mathcal J}}

\newcommand{\cM}{{\mathcal M}}
\newcommand{\cN}{{\mathcal N}}

\newcommand{\cU}{{\mathcal U}}

\newcommand{\CC}{{\mathbb C}}
\newcommand{\CP}{\mathbb{C} P}

\newcommand{\HH}{{\mathbb H}}

\newcommand{\PP}{{\mathbb P}}

\newcommand{\RR}{{\mathbb R}}
\newcommand{\RP}{\mathbb{R} P}
\newcommand{\BS}{{\mathbb S}}

\newcommand{\ZZ}{{\mathbb Z}}

\DeclareMathOperator{\im}{im}
\DeclareMathOperator{\coker}{coker}
\DeclareMathOperator{\id}{id}

\DeclareMathOperator{\mor}{mor}
\DeclareMathOperator{\nat}{nat}
\DeclareMathOperator*{\colim}{colim}
\DeclareMathOperator*{\holim}{holim}
\DeclareMathOperator*{\hocolim}{hocolim}
\DeclareMathOperator{\hofiber}{hofiber}
\DeclareMathOperator{\hocofiber}{hocofiber}

\DeclareMathOperator{\res}{res}
\DeclareMathOperator{\ind}{ind}
\DeclareMathOperator{\coind}{coind}
\DeclareMathOperator{\given}{\, | \,}

\DeclareMathOperator{\End}{End}

\DeclareMathOperator{\Sym}{Sym}

\DeclareMathOperator{\Gr}{Gr}
\DeclareMathOperator{\Map}{Map}

\DeclareMathOperator{\Th}{\mathbf \Theta}

\DeclareMathOperator{\Or}{O}
\DeclareMathOperator{\SO}{SO}
\DeclareMathOperator{\U}{U}
\DeclareMathOperator{\SU}{SU}

\DeclareMathOperator{\Sp}{Sp}
\DeclareMathOperator{\GL}{GL}

\DeclareMathOperator{\R}{R}

\DeclareMathOperator{\Top}{\underline{Top}}

\begin{document}

\pagenumbering{roman}


\thispagestyle{empty}

\vspace*{\fill}

\begin{center}
Equivariant Weiss Calculus and Loop Spaces of Stiefel Manifolds\\
\vspace{0.2in}
A dissertation presented\\
\vspace{0.2in}
by\\
\vspace{0.2in}
Philip Douglas Tynan\\
\vspace{0.2in}
to\\
\vspace{0.2in}
The Department of Mathematics\\
\vspace{0.2in}
in partial fulfillment of the requirements\\
for the degree of\\
Doctor of Philosophy\\
in the subject of\\
Mathematics\\
\vspace{0.2in}
Harvard University\\
Cambridge, Massachusetts\\
\vspace{0.2in}April 2016\\
\end{center}

\vspace*{\fill}

\pagebreak






\thispagestyle{empty} 

\vspace*{\fill}

\begin{center}
\copyright \, 2016 -- Philip Douglas Tynan \\
All rights reserved.
\end{center}

\vspace*{\fill}

\pagebreak




\doublespacing


\noindent Dissertation Advisor: Professor Hopkins \hfill
Philip Douglas Tynan

\vspace{0.5in}

\centerline{Equivariant Weiss Calculus and Loop Spaces of Stiefel Manifolds}

\vspace{0.8in}

\centerline{Abstract}

\vspace{0.3in}

In the mid 1980s, Steve Mitchell and Bill Richter produced a filtration of the Stiefel manifolds $\Or(V; W)$ and $\U(V; W)$ of orthogonal and unitary, respectively, maps $V \to V \oplus W$ stably split as a wedge sum of Thom spaces defined over Grassmanians. Additionally, they produced a similar filtration for loops on $\SU(V)$, with a similar splitting. A few years later, Michael Crabb made explicit the equivariance of the Stiefel manifold splittings and conjectured that the splitting of the loop space was equivariant as well. However, it has long been unknown whether the loop space of the real Stiefel manifold (or even the special case of $\Omega SO_n$) has a similar splitting.

Here, inspired by the work of Greg Arone that made use of Weiss' orthogonal calculus to generalize the results of Mitchell and Richter, we obtain an $\ZZ / 2 \ZZ$-equivariant splitting theorem using an equivariant version of Weiss calculus. In particular, we show that $\Omega \U(V; W)$ has an equivariant stable splitting when $\dim W > 0$. By considering the (geometric) fixed points of this loop space, we also obtain, as a corollary, a stable splitting of the space $\Omega ( \U(V; W), \Or(V_\RR; W_\RR) )$ of paths in $\U(V; W)$ from $I$ to a point of $\Or(V_\RR; W_\RR)$ as well. In particular, by setting $W = \CC$, this gives us a stable splitting of $\Omega (\SU_n / \SO_n)$.

\pagebreak





\tableofcontents

\pagebreak





\section*{Acknowledgements}

First and foremost, I would like to thank my dissertation advisor Mike Hopkins for his guidance and support. 

I would also like to thank the students and postdoctoral fellows at Harvard University for numerous helpful conversations. In particular, thank you to Omar Antolín Camarena, Gijs Heuts, Charmaine Sia, Jeremy Hahn, Krishanu Sankar, Danny Shi, Lukas Brantner, Akhil Matthew, Alex Perry, Peter Smillie, Francesco Cavazzani, Gabriel Bujokas, Nasko Atanasov, Simon Schieder, Justin Campbell, Anand Patel, Tom Lovering, Spencer Liang, Lucia Mocz, Adam Al-natsheh, Michael Fountaine, Aaron Landesman, Nat Mayer, Aaron Slipper, Mboyo Esole, and Hiro Tanaka!

For agreeing to be on my defense committee and taking the time to read this paper, I would like to thank my other readers Haynes Miller and Jacob Lurie.

I would not have been able to pursue a PhD in mathematics without the support of the faculty at MIT, my undergraduate alma mater. In particular, I would like to thank Haynes Miller, Mark Behrens, David Jerison, Paul Seidel, Katrin Wehrheim, Hamid Hezari, David Vogan, and Kiran Kedlaya.

The origins of my research were based on the work of Steve Mitchell, Bill Richter, Haynes Miller, Michael Crabb, Michael Weiss, and Greg Arone; so I would like to extend my gratitude to them as well, for making my thesis possible.

My family has been incredibly supportive of my academic pursuits, and I would like to take the time to give thanks to them.

Last but not least, I would like to thank Leah Brunetto for being a constant source of joy and inspiration.

\pagebreak




\pagenumbering{arabic}

\chapter{Introduction}

We consider the category $\cJ$ whose objects are finite dimensional real vector spaces and whose morphisms are the complex inner product preserving linear transformations between their complexifications, taken as a $G$ space where $G = \ZZ / 2 \ZZ$ acts by complex conjugation. Throughout this paper, we shall follow a similar strategy to that used in \cite{Arone}; in fact if we forget the $G$ action, we are working with the same spaces and spectra considered there. We will have to adapt many of the functor calculus tools to work in this equivariant setting, but luckily for us, most of the natural analogues that we will hold with minimal modifications.

Let $\cE$ be the category of functors $\cJ \to \Top$. Our criterion for polynomial functors in the $G$-equivariant setting will be reminiscent of that in section 6 of \cite{Weiss}. We will say that $F \in \cE$ is polynomial of degree $\le n$ if the natural map $F(V) \to T_n F(V) := \holim_{0 \ne U \subset \RR^{n + 1}} F(V \oplus U)$ is a $G$-equivalence. We find that the homotopy fiber of this map gives us the $(n+1)$th derivative $F^{(n + 1)}$, so unsurprisingly, when $F$ is polynomial of degree $\le n$, the functor $F^{(n + 1)}$ is $G$-contractible for every object in $\cJ$.

As one may expect, given any $F \in \cE$, we can use the natural transformation $T_n$, to construct a degree $\le n$ polynomial functor from $F$, which we shall refer to as its degree $n$ approximation. In particular, we will define $P_n F$ to be the homotopy colimit of the diagram $F \to T_n F \to T_n^2 F \to \cdots$. In some cases, this will allow us to split off parts of the space $F(V)$, for $V \in \cJ$, in the same way that the Taylor expansion in calculus let us split off the degree $n$ and smaller terms of an analytic function from the remaining ones. Following the same analogy, we may define a homogeneous polynomial of degree $n$ as one which is polynomial of degree $\le n$ and whose degree $n - 1$ approximation is $G$-contractible.

Much to our delight, the spectra associated to the derivatives of a homogeneous polynomial functor of degree $n$ are the same as those in the non-equivariant case, namely $G$-contractible aside from that of the $n$th derivative. Because the summands in the stable splitting are themselves homogeneous polynomial functors, this allows us to obtain a precise equivariant analogue of the splitting from Mitchell and Richter (and more generally, that from \cite{Arone}).

As is common in the study of loop groups and related homogeneous spaces, we replace $\Omega \U_{n, k}$ with a $\ZZ / 2$-homotopy equivalent space of algebraically defined loops, essentially using the strategy outlined in \cite{Pressley}. In particular, we replace our loop space with an infinite dimensional complex Grassmannian that is $\ZZ / 2$-equivalent. This allows us to define a filtration, like that from \cite{Crabb}, for which the associated graded pieces have a nice description, namely they become $G$-equivariant homogeneous polynomial functors after applying the functor $Q := \Omega^{\infty \rho} \Sigma^{\infty \rho}$ to them.

The $G$-equivariant splitting theorem that we obtain states that
$$F(V) \simeq \bigvee_{n > 0} (F_n(V) / F_{n - 1}(V)$$
for functors $F$ that have such a $G$-filtration $\{ F_n \}$ such that $F_0 = *$, and the homotopy cofiber of $Q F_{n - 1} (V) \to Q F_n(V)$ is a homogeneous polynomial functor. This means that the splitting of $\Omega \U(V; W)$ in \cite{Crabb} and \cite{Arone} is $\ZZ / 2$-equivariant, so in particular, we have a stable splitting of the $\ZZ / 2$ fixed points.

Moreover, we obtain a stable splitting of the geometric fixed points of $\Omega \U(V; W)$, which are given by the space of paths in $\U(V; W)$ that end at an element of $\Or(V_\RR; W_RR)$, denoted by $\Omega (\U(V; W), \Or(V_\RR; W_\RR)$. In particular, the real points stably split as the wedge sum of the real points of the summands $S_n^\RR(V)^{\hom(\tau, W_\RR)}$ appearing in the splitting of $\Omega \U(V; W)$. When $W = \CC$, we can identify $\U(V; \CC)$ with $\SU(V \oplus \CC)$. Because this has the structure of a group with $\Or(V_\RR; \RR) \cong \SO(V_\RR \oplus \RR)$ as a subgroup, this path space can be rewritten as $\Omega (\SU(V \oplus \CC) / \SO(V_\RR \oplus \RR))$.

\chapter{Equivariant Unitary Calculus}

\section{Functors on the Category $\cJ_n$}

Throughout this paper we will take $G = \ZZ / 2 \ZZ$. Here, we develop a $G$-equivariant version of the orthogonal calculus from \cite{Weiss}. We will use the notation $\cU$ for $\RR^\infty$ with the standard inner product, and regard all finite dimensional vector spaces as subspaces of $\cU$, inheriting its inner product. We will take the action of $G$ on $\CC$ to be that of complex conjugation, thereby making the complexification of any real vector space into a $G$-space. Note that this obviously gives us an identification of $\CC$ with the regular representation $\rho$ of $G$.

\begin{defn} For real vector spaces $V, W \subset \mathcal U$, let $\mor(V, W)$ be the set of complex linear transformations between the complexifications $V_\CC, W_\CC$, that preserve the induced Hermitian inner product. We will take $\cJ$ to be the category whose objects are finite dimensional subspaces of $\mathcal U$, with the set of morphisms from $V$ to $W$ given by the $G$-set $\mor(V, W)$. Note that the space of $G$-fixed points of $\mor(V, W)$ is simply the space of inner product preserving morphisms from $V$ to $W$, which we shall denote by $\mor^G(V, W)$.
\end{defn}

We will be interested in studying $G$-enriched functors $F: \cJ \to G-\Top$, the category of which we shall denote by $\cE$.

We can additionally define a vector bundle on the $G$-sets $\mor(V, W)$ in the following way.

\begin{defn}
Given $V \in \cJ$, we will denote $\CC^n \otimes_\CC V_\CC$ by $n V_\CC$.
For each nonnegative integer $n$, let $\gamma_n(V, W)$ be the bundle over $\mor(V, W)$ whose fiber over $\varphi \in \mor(V, W)$ is $n \cdot \coker(\varphi)$, glued together in the natural way. Note that we can naturally identify the cokernel of $\varphi$ with the orthogonal complement of $\varphi (V_\CC)$ in $W_\CC$. Furthermore if $v \in (\gamma_n(V, W))_\varphi = n \coker(\varphi)$, then $\overline v \in n \coker(\overline \varphi) = n \coker(g \varphi) = (\gamma_n(V, W))_{g \varphi}$, where $g$ is the generator of $G$, so this bundle inherits a $G$-space structure as well, where the $G$-fixed points are $\{ (\varphi, v) \ | \ \varphi \in \mor^G(V, W), v \in n \coker(\varphi) \subseteq n W \}$.
Define $\mor_n(V, W)$ as the Thom space of $\gamma_n(V, W)$, and $\mor_n^G(V, W)$ to be its $G$-fixed point space.
\end{defn}

Note that the space $\mor_n^G(V, W)$ is precisely what is referred to as $\mor_n(V, W)$ in \cite{Weiss} when the purely real analogue is considered.

The composition law $\mor(V, W) \times \mor(U, V) \to \mor(U, W)$ extends naturally to a vector bundle map $\gamma_n(V, W) \times \gamma(U, V) \to \gamma_n(U, W)$ given by $( (\varphi, w), (\psi, v) ) \mapsto ( \varphi \psi, w + \varphi_*(v) )$. On Thom spaces, we get a composition $\mor_n(V, W) \wedge \mor_n(U, V) \to \mor_n(U, W)$.

Letting $\cJ_n$ be the category with the same objects as $\cJ$ and with the set of morphisms from $V$ to $W$ given by $\mor_n(V, W)$, we obtain a pointed topological category. In fact, we obtain a pointed $G$-equivariant topological category, as we inherit the $G$-action from $\cJ$. Note that $\cJ_0$ is simply the pointed version of $\cJ$, obtained by adding a disjoint basepoint to each morphism set.

\begin{defn}
For $m \ge 0$, we will denote by $\cE_m$ the category of $G$-enriched functors $E: \cJ_m \to G-\Top_*$.
\end{defn}

We have an alternative description of $\cE_m$ as an element $E \in \cE_0$ together with a natural transformation $\sigma: S^{m V_\CC} \wedge E(W) \to E(V \oplus W)$ of functors on $\cJ_0 \times \cJ_0$ (where $m \cdot V_\CC$ can be regarded as the fiber over the canonical inclusion $\iota: W_\CC \hookrightarrow V_\CC \oplus W_\CC$ in the bundle $\gamma_m(W, V \oplus W)$, so this map arises from restricting the domain of the composition map $\mor_m(W, V \oplus W) \wedge E(W) \to E(V \oplus W)$). Note that the transformation $\sigma$ also satisfies an associativity condition by the functoriality of $E$.

\begin{defn}
The map $\sigma$ has an adjoint $\sigma^{ad}: E(W) \to \Omega^{m V_\CC} E(V \oplus W)$. We say that an object $E \in \cE_m$ is stable if $\sigma^{ad}$ is a homotopy equivalence for all $G$-representations $V, W$.
\end{defn}

In light of this, $\cE_1$ is equivalent to the category of $G$-spectra, as they both specify a space for each $G$-representation, and the natural map is exactly the required structure map. Similarly, the elements of $\cE_m$ for $m > 1$ can also be thought of as $G$-spectrum $\Th E$ without all representations being assigned a $G$-space; we define $\Th E$ as $\Th_{k \rho} := E(k \cdot \RR)$. Furthermore, a stable element of $\cE_m$ can be thought of as an $\Omega$-G-spectrum (of multiplicity $m$).

For positive $m$, we have the diagram $$E(V) \xrightarrow{\simeq} \hocolim_k \Omega^{mk \rho} E(k \cdot \RR \oplus V) \xleftarrow{\simeq} \hocolim_k \Omega^{mk \rho} ( S^{mV_\CC} \wedge E(k \cdot \RR) ),$$

where the left map comes from a diagram of homotopy equivalences from $\sigma^{ad}$ and the right map is $\hocolim_k \Omega^{mk \rho} \sigma$. The object on the right is simply $\Omega^\infty ( S^{m V_\CC} \wedge \Th E )$, where $\Omega^\infty \Theta := \hocolim_k \Omega^{k \rho} \Theta_{k \rho}$.

Conversely, if $\Th$ is any $G$-spectrum (made from well-pointed spaces), then we can define a functor $E: \cJ_0 \to G-\Top_*$ by $E(V) = \Omega^\infty ( S^{m V_\CC} \wedge \Th )$, that comes equipped with a natural transformation $\sigma: S^{m V} \wedge E(W) \to E(V \oplus W)$, and is therefore an element of $\cE_m$. Thus, we have an equivalence between stable objects of $\cE_m$ and $G$-spectra. This will become very important to us as the components appearing in our main splitting theorem will be of this form.

The following proposition will be useful in defining the derivatives of elements of $\cE_m$ in the next section.

\begin{prop} Let $V, W \in \cJ$. The reduced mapping cone of the restricted composition map $\mor_n(V \oplus \RR, W) \wedge S^{n \rho} \to \mor_n(V, W)$ is $G$-homeomorphic to $\mor_{n + 1}(V, W)$. It should be noted that the non-equivariant version of this is used in \cite{Arone}.
\end{prop}

\begin{proof} The non-equivariant real version of this is established as Proposition 1.2 in \cite{Weiss}, and it is clear that the proof works exactly the same way in the complex case, so the mapping cone will be homeomorphic to $\mor_{n + 1}(V, W)$.

That the restricted composition map is $G$-equivariant, along with the map
$$\mor_n(V, W) \to \mor_{n+1f}(V, W),$$
arising from the inclusion of $\CC^n$ into $\CC^{n + 1}$ given by $(v_1, \ldots, v_n) \mapsto (v_1, \ldots, v_n, 0)$, tells us that our homeomorphism is also $G$-equivariant, so we are done.
\end{proof}

\section{Derivatives of Functors}

Just as in the non-equivariant case, we can form inclusions $\cJ_0 \subset \cJ_1 \subset \cJ_2 \cdots$ and a notion of derivatives; if $E$ is a functor $\cJ_0 \to \Top_*$, then it has a derivative $E^{(1)}: \cJ_1 \to G-\Top_*$, which itself has a derivative $E^{(2)}: \cJ_2 \to G-\Top_*$, and so on. As in the non-equivariant case, the derivative is defined in terms of the adjoint to the restriction functor.

Restriction from $\cE_n$ to $\cE_{n-1}$ for $n > 0$ gives us a natural transformation $\res_{n-1}^n$, and more generally we can obtain a restriction map $\res_m^n: \cE_n \to \cE_m$ for $m \le n$ by successive compositions.

The restriction transformation $\res_m^n$ has a right adjoint, which we shall denote $\ind_m^n$. As in \cite{Weiss}, we can use Yoneda's lemma to work out a more useful definition of $\ind$. We have $\ind_m^n E(V) = \nat_n( \mor_n(V, - ), \ind_m^n E ) = \nat_m( \res_m^n \mor_n(V, - ), E )$.

For $m \le n, E \in \cE_m$, we will define the derivatives of $E$ by $E^{(n - m)} := \ind_m^n E$ (although some might argue that this should really be called $\coind$. In particular, this means that for $E \in \cJ_0$, the derivatives are given by $E^{(n)} = \ind_0^n E$.

As in the non-equivariant case detailed in \cite{Weiss}, we would like a more concrete description of the $\ind$ functors. Recall that for $V \in cJ_n$, we have the $G$-homotopy cofiber sequence $\mor_n(V \oplus \RR, -) \wedge S^{n \rho} \to \mor_n(V, -) \to \mor_{n + 1}(V, -)$ of functors in $\cE_n$. For $F \in \cE_n$, we can dualize to obtain the $G$-homotopy fiber sequence
$$\nat_n(\mor_{n + 1}(V, -), F) \to \nat_n(\mor_n(V, -), F) \to \nat_n(S^{n \rho} \wedge \mor_n(V \oplus \RR, -), F).$$

The Yoneda lemma tells us that this is equivalent to the homotopy fiber sequence
$$\res_n^{n + 1} \ind_n^{n + 1} F(V) \to F(V) \xrightarrow{\sigma^{ad}} \Omega^{n \rho} F(V \oplus \RR).$$

Thus, we obtain an explicit description of $F^{(1)}(V)$, and perhaps more importantly, one that reminds us of the derivative from Newtonian calculus, as $F^{(1)}$ measures the difference between $F(V)$ and $F(V oplus \RR)$ in some sense (and $\RR$ is the smallest increment available to us here).

\section{Unitary Actions}

For each non-negative integer $n$, the category $\cJ_n$ is equipped with a left action by the unitary group $\U_n$. As you may have already guessed, this is where the ``unitary" in ``unitary calculus" comes from. Since any continuous functor $E \in \cE_0$ has an $n$th derivative $E^{(n)} \in \cE_n$, this action is intimately related to the derivatives of the functors. We should point out here that because the group $G$ acts on the unitary group via conjugation as well, what we really have is a $\U_n \rtimes G$ action.

Like the orthogonal action described in \cite{Weiss}, the unitary action leaves the objects of $\cJ_n$ fixed and acts on morphisms via its action on $\CC^n$. We will denote the automorphism of $\cJ_n$ corresponding to $t \in \U_n$ by $\lambda(t)$.

This gives us a left action of $\U_n$ on the objects and morphisms of $\cE_n$ as well. For objects $E \in cE_n$, this action is given by $t E := E \circ \lambda(t)^{-1}$ for $t \in \U_n$. For morphisms, we have the trivial action $t \eta = \eta$, for $\eta \in \nat_n(E, F)$. In diagram form, if we have a natural transformation $\eta: E \to F$, and a morphism $\varphi \in \mor_n(V, W)$, then after applying $t \in U_n$ we obtain
$$\begin{array}{ccc}
tE(V) & \xrightarrow{\eta_V} & tF(V) \\
\big{\downarrow}{E(t^{-1} \varphi)} & & \big{\downarrow}{E(t^{-1} \varphi)} \\
tE(W) & \xrightarrow{\eta_W} & tF(W)
\end{array}$$

To examine how this orthogonal action relates to derivatives, we consider the pair of adjoint functors $\res_0^n, \ind_0^n$. We must have $\res_0^n (t E) = \res_0^n E, \res_0^n (t \eta) = \res_0^n \eta$, for any $\eta \in \nat_n(E, F)$, as the action over $\cE_0$ is trivial.

We know that the functor $F \mapsto \ind_0^n F = F^{(n)}$ is a right adjoint of $\res_0^n$. However, we can clearly see that the functor $F \mapsto t (F^{(n)})$ (for a fixed $t \in \U_n$) is also a right adjoint to $\res_0^n$. By the uniqueness of right adjoints, we have a unique isomorphism $\alpha_t: F^{(n)} \to t ( F^{(n)} )$ such that $u_F \alpha_t = u_F$, where $u_F: \res_0^n F^{(n)} \to F$ is the universal morphism.

From this, we obtain the following proposition:

\begin{prop} There exists a family $\{ \alpha_V \}: \U_n \times F^{(n)}(V) \to F^{(n)}(V)$ making the maps $ev: \mor_n(V, W) \wedge F^{(n)}(V) \to F^{(n)}(W), u: F^{(n)}(V) \to F(V)$ into $\U_n \rtimes G$-maps.
\end{prop}

The Yoneda definition of the derivative, $F^{(n)}(V) = \nat_0( \mor_n(V, -), F )$ gives us an explicit definition of $\alpha_V$.

We would like to study this property, and objects of $\cE_n$ that exhibit it, including but not necessarily only the $n$th derivatives.

\begin{defn} A symmetric object in $\cE_n$ is an object $E$ together with continuous actions $\U_n \rtimes G \times E(V) \to E(V)$ for each $V \in \cJ_n$, such that for each $W \in \cJ_n$, the evaluation map $\mor_n(V, W) \wedge E(V) \to E(W)$ is a $\U_n \rtimes G$ map.
\end{defn}

\chapter{Polynomial Functors}

\section{Computing Higher Derivatives}

We begin this chapter by crafting a $G$-equivariant analogue of the notion of polynomial functor. We will need to first consider another category into which we may embed $\cJ$.

\begin{defn}
Let $\cJ$ be the category whose objects are the finite dimensional complex vector subspaces of $\cU_\CC := \cU \otimes \CC$, considered as an inner product space with the standard Hermitian inner product, and whose morphism spaces are given by the inner product preserving complex linear transformations. If $V, W \in \cJ_\CC$, we will denote the morphisms in $\cJ_\CC$ by $\tilde \mor(V, W)$.
\end{defn}

It is clear that we have an inclusion of categories $\cJ \hookrightarrow \cJ_\CC$ given by sending $V$ to $V_\CC$. We may extend the $G$ action to $\cJ_\CC$, but it is no longer the case that all objects are fixed by the generator $g \in G$. Still, it is clear that $g$ acts as an involution on both objects and morphisms. Furthermore, the objects of $\cJ_\CC$ that are sent to themselves under the action of $G$ are precisely those in the image of $\cJ$.

This allows us to define $\cE_\CC$ as the category of $G$-enriched functors $E: \cJ_\CC \to G-\Top$, and the functor $\cJ \to \cJ_\CC$ gives rise to a dual functor $\cE_\CC \to \cE$. For all of our objects $F \in \cE$ of interest in this paper, there exists a canonical lift of $F$ in $\cE_\CC$, which we shall denote by $\tilde F$.

We now make the following definition for an indexing category that we will be using frequently.

\begin{defn}
Fix a nonnegative integer $n$, and let $\cC_n$ be the category whose objects are nonzero complex subspaces $U \subseteq \CC^{n + 1}$ with the only maps allowed being the inclusions. By thinking of $\cC_n$ as a subcategory of $\cJ_\CC$, and noting that it is closed under the given $G$ action, $\cC_n$ inherits a $G$ action which is an involution on its objects.
\end{defn}

We now present the following theorem, the non-equivariant version of which is used in \cite{Arone}, that will assist us in computing higher derivatives of functors.

\begin{thm}
Let $V, W \in \cJ$, and define $M_n(V, W) := \hocolim_{U \in \cC_n^{op}} \tilde \mor( U \oplus V_\CC, W_\CC )$, where $\cC_n$ is as defined above. The unreduced mapping cone of the restriction map
$$M_n(V, W) \to \tilde \mor(V_\CC, W_\CC) = \mor(V, W)$$
is $G$-homeomorphic to $\mor_{n + 1}(V, W)$. Furthermore, this homeomorphism is a natural transformation of functors on $\cJ^{op} \times \cJ$.
\end{thm}

\begin{proof}
The real case is proven in \cite{Weiss}, and it is clear that the same proof works in the complex case. In particular, the setup in the latter case uses the same underlying spaces as we do here and an equivalent indexing category (as noted above). Because of this, we know that non-equivariantly, our mapping cone is homeomorphic to that one, namely $\mor_{n + 1}(V, W)$. Thus, it suffices to show that the given homeomorphism is $G$-equivariant.

It is clear that the map from $\mor(V, W)$ to the given mapping cone is a $G$-equivariant, as $M_n(V, W) \to \mor(V, W)$ is $G$-equivariant. This tells us that the base space of the mapping cone is $G$-homeomorphic to $\mor(V, W)$, so we need only show that the fibers have the same $G$ action as those of $\mor_{n + 1}(V, W)$. Over $\varphi \in \mor(V, W)$, the former arise from complementary subspaces of $V_\CC$ in $W_\CC$ which are also subspaces of $\CC^{n + 1}$, while the latter is $(n + 1) \coker(\varphi)$. It is evident by construction that both representations inherit their $G$ action from that of $\CC^{n + 1}$ arising from complex conjugation, and therefore are $G$ homeomorphic. The result follows.
\end{proof}

As in \cite{Weiss}, we wish to consider polynomial functors. Our criterion will be essentially the same as before, but we must pay close attention to the indexing category, as it now receives an action from $G$. For a given functor $F: \cJ \to G-\Top$ define
$$T_n F(V) := \holim_{U \in \cC_n} \tilde F(V_\CC \oplus U ).$$

\begin{defn}
We say that a functor $F \in \cE$ is polynomial of degree $\le n$ if the natural map $F(V) \to T_n F(V)$ is a homotopy equivalence for all $V \in \cJ$.
\end{defn}

\begin{prop}
Suppose $E \in \cE_0$ and $V \in \cJ$. Then, the sequence
$$ E^{(n + 1)} (V) \xrightarrow u E(V) \xrightarrow p T_n E (V)$$
is a $G$-fibration sequence.
\end{prop}

\begin{proof}
The map $p$ can be rewritten using the Yoneda lemma as
\begin{align*}
\nat_0 (\mor_0( V, - ), E) = \nat_0 (\tilde \mor_0( V_\CC, - ), E) & \to \holim_{U \in \cC_n} \nat_0 (\tilde \mor_0( U \oplus V_\CC, -), E) \\
& = \nat_0( \hocolim_{U \in \cC_n^{op}} \tilde mor_0( U \oplus V_\CC, - ), E) \\
& = \nat_0( M_n(V, -)_+, E),
\end{align*}
which arises from the natural transformation $M_n(V, -)_+ \to \mor_0(V, -)$.

We also know from the theorem at the end of the previous section that the reduced mapping cone of the map $M_n(V, W)_+ \to \mor_0(V, W)$ is $\mor_{n + 1}(V, W)$. By mapping to $E$, we see that the homotopy fiber of the original map is $\nat_0( \mor_{n + 1}( V, - ), E)$, which is simply the definition of $E^{(n + 1)}(V)$. Thus, we have a $G$-fibration sequence.
\end{proof}

\begin{lem}
Let $\eta: E \to F$ be a morphism in $\cE_0$, and suppose that $E$ is polynomial of degree $\le n$ and that $F^{(n + 1)}$ vanishes. Then, the functor $V \mapsto \hofiber [ E(V) \xrightarrow \eta F(V) ]$ is also polynomial of degree $\le n$.
\end{lem}

\begin{proof}
This follows immediately from the applying the above proposition.
\end{proof}

We have two important corollaries of this lemma.

\begin{cor}
If $E, F$ are polynomial of degree $\le n$, and we have a natural transformation $\eta: E \to F$, then the homotopy fiber of $\eta$ is polynomial of degree $\le n$ as well.
\end{cor}

\begin{cor}
If $F \in \cE_0$ and $F^{(n + 1)}$ vanishes, then $\Omega F$ is polynomial of degree $\le n$.
\end{cor}

\begin{proof}
Set $E \equiv *$. Here, the functor $V \mapsto \hofiber[E(V) \to F(V)]$ is simply $V \mapsto \Omega F(V)$.
\end{proof}

We make the following remark here.

\begin{rmk}
For functors which have a de-looping, such as infinite loop spaces, the above corollary gives us a simple criterion for being polynomial of degree $\le n$, namely that of the $n+1$th derivative vanishing. This means that such functors follow some of our Newtonian calculus intuition about the functor calculus presented here.
\end{rmk}

\section{Approximation by Polynomial Functors}

So far, we have been studying the calculus of functors from the Newtonian viewpoint- computing $n$th derivatives, and making statements about the vanishing of higher order derivatives in order to classify our functors of interest. However, one knows that ordinary calculus can be viewed through an alternative lens; namely that concerned with polynomial approximations of functions. Likewise, one can study functors on $\cJ$ in this manner as well.

\begin{defn}
Let $F \in \cE$, and consider the natural map $p: F \to T_n F$. We define the $n$th degree polynomial approximation of $F$ to be the homotopy colimit of the diagram
$$F \xrightarrow p T_n F \xrightarrow p \to T_n^2 \xrightarrow p \cdots$$
\end{defn}

It is clear that $P_n F$ is polynomial of degree $\le n$, as the map
$$P_n F \to T_n P_n F = T_n \hocolim_k T_n^k F = \hocolim_k T_n^{k + 1} F \simeq P_n F$$
is obviously a $G$-equivalence. \\

Secondly, the maps $p: F \to T_n F$ induce a natural map $F \to P_n F$. If $F$ is polynomial of degree $\le n$, then each $p$ is by definition a $G$-equivalence, and therefore the map $F \to P_n F$ is a $G$-equivalence as well. \\

Lastly, the construction $P_n F$ has another useful property, that of being an approximation of degree $n$. \\

\begin{defn}
A morphism $\eta: E \to F$ in $\cE_0$ is an approximation of order $m$ if the maps
$$\begin{array}{llll}
\eta_*: E(\RR^\infty) \to F(\RR^\infty) \\
\eta_*: \Th E^{(i)} \to \Th F^{(i)} &&& \text{for } 1 \le i \le m
\end{array}$$
are $G$-equivalences.
\end{defn}

To prove this, we will need the following lemma. \\

\begin{lem}
Suppose $E \in \cE$ and $q$ is a nonnegative integer. Then, the natural map $p_*: E \to T_q E$ is a $G$-equivalence at infinity. In particular, if $E$ is $G$-connected at infinity, then so is $T_q E$.
\end{lem}

\begin{proof} The proof is similar to the analogous result in \cite{Weiss}. As is the case there, the basic idea here is that the codomain is a homotopy colimit of homotopy limits, which we shall express instead as a homotopy limit of homotopy colimits (in particular those that look like the homotopy colimit defining the domain).

We have
\begin{align*}
(T_q E)(\RR^\infty) & = \hocolim_n T_q E(\RR^n) \\
& = \hocolim_n \holim_{U \in \cC_q} \tilde E(\CC^n \oplus U) \\
& = \hocolim_n \holim_{U \in \cC_q} \tilde E( \CC^n \oplus U ) \\
& \simeq \hocolim_n \holim_{U \in \cC_q} \hocolim_{k \le n} \tilde E( \CC^k \oplus U ) \\
& \simeq \colim_n \holim_{U \in \cC_q} \hocolim_{k \le n} \tilde E( \CC^k \oplus U ) \\
& \cong \holim_{U \in \cC_q} \hocolim_{k \le \infty} \tilde E( \CC^k \oplus U )
\end{align*}, because the nerve of the poset that the homotopy limit is taken over is compact. Because taking fixed points commutes with homotopy colimits, these are all equivalences on $G$-fixed points as well.

Now, we have
\begin{align*}
\holim_{U \in \cC_q} \hocolim_{k \le \infty} \tilde E( \CC^k \oplus U ) & \cong \holim_{U \in \cC_q} \hocolim_k \tilde E( \CC^k ) \\
& = \holim_{U \in cC_q} \tilde E( \CC^\infty ) \\
& \simeq \tilde E( \CC^\infty ) \\
& = E(\RR^\infty),
\end{align*}
with the second to last equivalence arising from the fact that the poset has a maximal element, which clearly gives rise to an equivalence for the $G$-fixed points as well.

It is also evident that this last equivalence is the inverse of $p_*$.
\end{proof}

We are now ready for the following proposition.

\begin{prop} Let $F \in \cE_0$. The natural map $\psi: F \to P_n F$ is an approximation of order $n$.
\end{prop}

\begin{proof}: We will proceed in three main steps.

1) We first establish that $F \to T_n F$ is an approximation of order $n$.

The previous lemma tells us that $p_*: E(\RR^\infty) \to (T_n E)(\RR^\infty)$ is a $G$-equivalence. Furthermore, we know that the homotopy fiber of $p: E \ T_n E$ is $E^{(n + 1)}$, and therefore the $G$-spectrum
$$\Th \hofiber(p)^{(i)} = \Th [ \ind^i_0 \res^{n + 1}_0 \ind^{n + 1}_0 E ]$$
is $G$-contractible for $1 \le i \le n$.

Thus, we have that $p_*: \Th E^{(i)} \to \Th (T_n E)^{(i)}$ is a $G$-equivalence for $1 \le i \le n$, so $p$ is an approximation of order $n$.

2) We claim that all maps in the diagram $F \xrightarrow p T_n F \xrightarrow{T_n p} T_n^2 F \to \cdots$ are approximations of order $n$.

Suppose for the sake of induction that we know this is true when the codomain is $T_n^k E$. Then, consider the diagram
$$\begin{array}{ccc}
T_n^k E & \xrightarrow{T_n^k p} & T_n^{k + 1} E \\
\downarrow p & & \downarrow p \\
T_n^{k + 1} E & \xrightarrow{T_n^{k + 1} p} & T_n^{k + 2} E
\end{array}$$

The previous step tells us that the vertical arrows are approximations of order $n$, and the inductive hypothesis tells us that the top horizontal arrow is as well. Thus, since plugging in $\RR^\infty$ and computing associated spectra (for degrees $1$ through $n$) give us analogous diagrams with every map a $G$-equivalence, we see that the bottom horizontal arrow is also an approximation of order $n$.

The previous step gives us our base case of $k = 1$, so we may conclude that each map is an approximation of order $n$.

3) From the previous step, for each map in the diagram defining $P_n E$, all of the induced maps on $\RR^\infty$ and the maps of associated spectra (for degrees $1$ through $n$) are $G$-equivalences, and therefore the maps
$$\begin{array}{llll}
\psi_*: E(\RR^\infty) \to (P_n E)(\RR^\infty) \\
\psi_*: \Th E^{(i)} \to \Th (P_n E)^{(i)}, &&& \text{for } 1 \le i \le n
\end{array}$$
are $G$-equivalences. We conclude map $F \to \hocolim_k T_n^k F = P_n F$ is an approximation of order $n$.
\end{proof}

\begin{cor}
For any $E \in \cE_0$, the map $P_n E \xrightarrow{P_n \psi} P_n P_n E$ is a $G$-equivalence.
\end{cor}

\begin{proof}Consider the diagram
$$\begin{array}{ccc}
E & \xrightarrow \psi & P_n E \\
\big \downarrow \psi & & \big \downarrow \psi \\
P_n E & \xrightarrow{P_n \psi} & P_n P_n E
\end{array}$$

Three of the arrows are approximations of order $n$, so we may conclude that the fourth is as well.
\end{proof}

\section{Behavior at Infinity}

Now that we have a nice collection of theorems about $G$-equivariant polynomial functors, we can say a bit more about how these functors behave. For starters, polynomial functors are determined by their behavior at infinity.

To make this more precise, if $E \in cE$ is polynomial of degree $\le n$, then the natural map $E \to P_n E = \hocolim_k T_n^k E$ is a $G$-equivalence. However, we can easily see that $T_n^k E$ depends only on the behavior of vector spaces of dimension at least $k |G|$ and morphisms between them. Thus, this homotopy colimit is ultimately determined by the behavior of $E$ on large representations of $G$.

Def: An object $E \in \cE$ is said to be $G$-connected at infinity if the object $E( \RR^\infty ) = E( \cU )$ is $G$-connected.

\begin{lem} Suppose that $E \in \cE_0$ is polynomial of degree $\le n$ and $G$-connected at infinity. For $V \in \cJ_0$, define $E_\flat (V)$ to be the base point component of $E(V)$. Then, $P_n E_\flat$ is $G$-equivalent to $E$. Thus, $E_\flat$ determines $E$, up to $G$-equivalence.
\end{lem}

\begin{proof} Because $E$ is polynomial of degree $\le n$, the inclusion $E \to P_n E$ is a $G$-equivalence. We also have an inclusion $\iota: P_n E_\flat \to P_n E$, which we claim is a $G$-equivalence.

Because $E$ is connected at infinity, it is sufficient to show that the homotopy fiber of $\iota: P_n F_\flat(V) \to P_n F(V)$ is $G$-contractible for all $V \in cJ$. We already know that each such fibe must be either $G$-contractible or empty, as $T_n^k E_\flat (V) \hookrightarrow T_n^k E(V)$ is the inclusion of a union of connected components for all $k \ge 0$.

Next, we will show that no fiber is empty. Let $C$ be a connected component of $P_n F(V)$, and let $C' = C \cap F(V)$, which we know to be nonempty as the inclusion $F(V) \hookrightarrow P_n F(V)$ is a homotopy equivalence. We can choose a $k$ such that $C'$ maps to the base point component of $F(V \oplus \RR^k)$ via the inclusion induced map $F(V) to F(V \oplus \RR^k)$ because $F$ is $G$-connected at infinity. This tells us that the image of $C'$ in $T_n^k F(V)$ is contained in the image of $T_n^k F_\flat (V)$, so the image of $\iota$ has nonempty intersection with $C$. Thus, the homotopy fiber over $C$ is nonempty, and therefore contractible. The result follows.
\end{proof}

\begin{prop} Let $\eta: E \to F$ be a morphism of polynomial functors in $\cE_0$ such that the homotopy fiber of $\eta$ is contractible for all $V \in \cJ$. If $F$ is connected at infinity, then $\eta$ is a $G$-equivalence.
\end{prop}

\begin{proof} The condition on $\eta$ means that $\eta_\flat: E_\flat \to F_\flat$ is a $G$-equivalence. Additionally, we know that $E$ must be connected at infinity, as the homotopy fiber of $\eta_*: E(\RR^\infty) \to F(\RR^\infty)$ is contractible and $F$ is connected at infinity. This means that $E, F$ are completely determined by $E_\flat, F_\flat$, so exhibiting the $G$-equivalence $\eta_\flat$ is sufficient.
\end{proof}

\begin{lem} If $E \in \cE$ is polynomial of degree $\le m$, then $T_n E$ is also polynomial of degree $\le m$, for any nonnegative integer $n$.
\end{lem}

\begin{proof} Consider the canonical map
\begin{align*} p: T_n E(V) \to T_m T_n E(V) & = \holim_{W \in \cC_m} \holim_{U \in \cC_n} \tilde E(V_\CC \oplus U \oplus W) \\
& = \holim_{U \in \cC_n} \holim_{W \in \cC_m} \tilde E(V_\CC \oplus U \oplus W) \\
& = T_n T_m \tilde E(V).
\end{align*}
Thus, $p$ arises as $T_n$ of the map $E(V) \to T_m E(V)$, which we know to be a $G$-equivalence, and therefore $p$ is a $G$-equivalence. We conclude that $T_n E$ is polynomial of degree $\le m$.
\end{proof}

\begin{cor} Let $\eta: E \to F$ be a morphism in $\cE_0$ such that $E, F$ are polynomial of degree $\le n$, and $F$ is connected at infinity. Suppose that the maps
$$\begin{array}{llll}
\eta_*: E(\cU) \to F(\cU) \\
\eta_*: \Th E^{(i)} \to \Th F^{(i)} &&& for \, \, 1 \le i \le n
\end{array}$$
are $G$-homotopy equivalences. Then, $\eta$ is a $G$-equivalence.
\end{cor}

\begin{proof} We proceed in a similar manner to the proof of the analogous statement in \cite{Weiss}. By the previous proposition, it is sufficient to show that $D = \hofiber(\eta)$ vanishes. By a previous lemma, we know that $D$ is polynomial of degree $\le m$, so $D$ is polynomial of degree $k$, for some integer $0 \le k \le m$. This means that $D^{(k)}$ is stable, so we have $D^{(k)} \simeq \Omega^\infty [ S^{k V} \wedge \Th D^{(k)} ]$. However, $\Th D^{(k)}$ is contractible by assumption, and therefore $D^{(k)}$ must vanish.

We know that we can also express $D^{(k)}$ as the homotopy fiber of the natural map $p: D \to T_{k - 1} D$, both of which are polynomial of degree $\le k$. Additionally, we know that $D(\RR^\infty) \hofiber [ E(\RR^\infty) \to F(\RR^\infty) ]$, which is contractible by assumption. By the previous lemma, this means that $T_q D$ is contractible at infinity as well, and therefore $p$ is a $G$-equivalence.

By definition, this means that $D$ is polynomial of degree $\le k - 1$, contradicting the minimality of $k$. From this, we may conclude that $D$ is polynomial of degree $0$, and therefore that $D$ is stable. This means that $D(V) \simeq D(\RR^\infty)$ for all $V \in \cJ$, and therefore is contractible for all $V \in \cJ$. We conclude that $\eta$ is a $G$-equivalence.
\end{proof}

From this, we obtain an important corollary.

\begin{cor} Let $E \in \cE$, and suppose that $\Th E^{(i)}$ is contractible for $1 \le i \le m$. Then, $P_m E (V)$ is contractible for all $V \in \cJ$.
\end{cor}

This will be useful in assessing whether a given polynomial functor is homogeneous, as we shall see in the next section.

\section{Homogeneous Polynomial Functors}

In \cite{Weiss}, a homogeneous polynomial functor is defined as a functor $F$ which is polynomial of degree $n$ such that $P_{n-1} F$ is contractible. Intuitively, we can think of homogeneous polynomial functors as the basic building blocks of polynomial functors in the same way that homogeneous polynomials over a field are the basic building blocks of polynomials over that field. In fact, 

We already know that non-equivariantly, if $F$ is a polynomial functor of degree $n$, then $F^{(n)}$ is stable and symmetric.

Let us now look at an example base on Example 5.7 from \cite{Weiss}.

\begin{eg} Consider the functor $F: \cJ_0 \to G-\Top_*$, defined by
$$F(\tilde V) := \Omega^\infty [ (S^{n V} \wedge \Th)_{h \U_n} ],$$
where $V = \tilde V_\CC$ for the sake of notational simplicity, and $\Th$ is a $G$-spectrum that additionally has an action of $\U_n$ compatible with this. By this, we mean that the following diagram commutes: $$\begin{array}{ccc}
(\U_n)_+ \wedge \Th & \xrightarrow{\mu} & \Th \\
\downarrow{g} & & \downarrow{g} \\
(\U_n)_+ \wedge \Th & \xrightarrow{\mu} & \Th
\end{array}$$ where $g$ is the nontrivial element of $G$.
\end{eg}

As in the referenced example, we shall show that $F^{(n+1)}$ vanishes by identifying the sequence of derivatives $F^{(n)} \xrightarrow u F^{(n-1)} \xrightarrow u \cdots \xrightarrow u F^{(1)} \xrightarrow u F$ (where the maps come from the homotopy fibration $F^{(i + 1)}(\tilde V) \xrightarrow u F^i(\tilde V) \to \Omega^{i \rho} F^i(\tilde V \oplus \RR)$) with another sequence, $F[n] \hookrightarrow F[n-1] \hookrightarrow \cdots \hookrightarrow F[1] \hookrightarrow F$, where $F[i](\tilde V) := \Omega^\infty [ (S^{n V} \wedge \Th)_{h \U_{n - i}} ]$ where $\U_{n - i} < \U_n$ is the subgroup fixing the first $i$ coordinates. In doing so, we also gain some insight into what taking each derivative of a homogeneous polynomial functor actually does.

We first observe that $F[i] \in \cE_i$, as we have a natural transformation
$$\sigma: S^{i V} \wedge F[i](\tilde W) \to F[i](\tilde V \oplus \tilde W)$$
given by the inclusion map
\begin{align*}
S^{i V} \wedge \Omega^\infty [ (S^{n W} \wedge \Th)_{h \U_{n - i}} ] & \hookrightarrow \Omega^\infty [ S^{i V} \wedge (S^{n W} \wedge \Th)_{h \U_{n - i}} ] \\
& \xrightarrow{=} \Omega^\infty [ (S^{i V} \wedge S^{n W} \wedge \Th)_{h \U_{n - i}} ] \\
& \hookrightarrow \Omega^\infty [ (S^{n V} \wedge S^{n W} \wedge \Th)_{h \U_{n - i}} ].
\end{align*}
It is also clear that this map is $G$-equivariant.

It is clear that $F[n]$ is a stable object in $\cE_n$, as the natural map
$$\sigma^{ad}: \Omega^\infty (S^{n W} \wedge \Th) \to \Omega^V\Omega^{\infty} (S^{n (V \oplus W)} \wedge \Th) = \Omega^{\infty} \Omega^V \Sigma^V (S^{n W} \wedge \Th)$$
is a $G$-homotopy equivalence (by the definition of $\Omega^\infty$). Thus, it is clear that $F[n]^{(1)}(\tilde V)$ is contractible for all $\tilde V \in \cJ$, so it will be sufficient to identify $F[n]$ with $F^{(n)}$.

The inclusion $F[i + 1] \hookrightarrow F[i]$ is really a map in $\cE_i$, $\res_i^{i+1} F[i + 1] \to F[i]$, which has an adjoint in $\cE_{i + 1}$, $F[i + 1] \to \ind_i^{i + 1} F[i] = F^{(1)}[i]$. We wish to show that this map is a canonical $G$-homotopy equivalence, as this would tell us that the map $F[i + 1] \hookrightarrow F[i]$ can be identified with $u: F^{(i + 1)} \to F^i$.

We know that
\begin{align*}
F[i]^{(1)}(\tilde V) & := \hofiber [ F[i](\tilde V) \to \Omega^{i \rho} F[i](\tilde V \oplus \RR) ] \\
& = \hofiber [ \Omega^\infty (S^{n V} \wedge \Th)_{h \U_{n - i}} \to \Omega^{\infty + i \rho} (S^{n (V \oplus \rho)} \wedge \Th)_{h \U_{n - i}} ].
\end{align*}
Because this is an infinite loop map, this is $\Omega^\infty$ of the homotopy fiber of
\begin{align*}
(S^{n V} \wedge \Th)_{h \U_{n - i}} \to \Omega^{i \rho} (S^{n (V \oplus \rho)} \wedge \Th)_{h \U_{n - i}} & \cong \Omega^{i \rho} (S^{i \rho} \wedge S^{(n - i) \rho} \wedge S^{n (V)} \wedge \Th)_{h \U_{n - i}} \\
& \cong \Omega^{i \rho} \Sigma^{i \rho} (S^{(n - i) \rho} \wedge S^{n (V)} \wedge \Th)_{h \U_{n - i}}.
\end{align*}

Since $\Omega^\infty$ of the last term is equal to $\Omega^\infty (S^{(n - i) \rho} \wedge S^{n (V)} \wedge \Th)_{h \U_{n - i}}$, we see that $F[i]^{(1)}(\tilde V)$ is given by $\Omega^\infty$ of the homotopy fiber of
$$(S^{n (V)} \wedge \Th)_{h \U_{n - i}} \to (S^{(n - i) \rho} \wedge S^{n (V)} \wedge \Th)_{h \U_{n - i}}.$$

This map arises from taking $\U_{n - i}$ homotopy orbits of the $\U_{n - i} \rtimes G$-equivariant map $\varphi: S^{n (V)} \wedge \Th \to S^{(n - i) \rho} \wedge S^{n (V)} \wedge \Th$. Furthermore, $\varphi = \iota \wedge \id_{S^{n V} \wedge \Th}$, where $\iota: S^0 \to S^{(n - i) \rho}$, which has homotopy fiber equal to $S^{(n - i) \rho - 1}_+$. This tells us that $\varphi$ has homotopy fiber $(S^{(n - i) \rho - 1}_+) \wedge S^{n V} \wedge \Theta$, so our homotopy fiber of interest is $\Omega^\infty [ ( (S^{(n - i) \rho - 1}_+) \wedge S^{n V} \wedge \Th )_{\U_{n - i}} ] \cong \Omega^\infty [ ( S^{n V} \wedge \Th )_{\U_{n - i - 1}} ] = F[i + 1](\tilde V)$. It is clear that all of these $G$-homotopy equivalences are natural, and therefore $F[i + 1](\tilde V) \to F[i](\tilde V) \to \Omega^{i \rho} F[i](\tilde V \oplus \RR)$ is a $G$-homotopy fibration.

Thus, we can conclude that $F^{(n)}(\tilde V)$ is stable, and therefore that $F$ is polynomial of degree $\le n$.

That $F$ is also homogeneous (of degree $n$) follows from computing the associated spectra of $F^{(i)}$ for each $i$. This is almost exactly as done in \cite{Weiss}. Firstly, it is clear that $\Th (F^{(i)}) = *$ for all $i > n$, and by the stability calculation, that $\Th ( F^{(n)}) ) = \Th (F[n]) = \Th$. Finally, for $i < n$, we have $\Th ( F^{(i)} ) \simeq \Th (F[i]) \simeq [\Th (\res_i^n F[n])]_{h \U_{n - i}}$. We claim that $\Th ( \res_i^n D )$ is contractible for any $D \in \cE_n$, from which it would follow that $\Th (F^{(i)})$ is contractible.

To see this, consider $\pi_V^H (\res_i^n D) = \hocolim_k [S^{k i \rho} \wedge S^V, D(k \RR)]^H$, where the maps in the diagram are given by
\begin{align*}
[S^{k i \rho} \wedge S^V, D(k \RR)]^H & \to [S^{i \rho} \wedge S^{k i \rho} \wedge S^V, S^{i \rho} \wedge D(k \RR)]^H \\
& \xrightarrow{\iota} [S^{(k + 1) i \rho} \wedge S^V, S^{n \rho} D(k \RR)]^H \\
& \xrightarrow{\sigma} [S^{(k + 1) i \rho} \wedge S^V, D((k + 1) \RR)]^H,
\end{align*}
where the last map is the structure map of $D$ as an element of $\mathcal E_n$. It is evident that $\iota$ is $G$-nullhomotopic, as the inclusion $S^{i \rho} \hookrightarrow S^{n \rho}$ is.

Now, consider the trivial map of functors given by $P_{n - 1} F \to *$. It is obvious that the target is $G$-connected at infinity, and that both functors are polynomial of degree $\le n$. By the above computation, we have a $G$-equivalence $\Th^{(i)} (P_{n - 1} F) \to \Th^{(i)} * = *$ for $1 \le i \le n - 1$. Thus, we may conclude that this map is an equivalence, and therefore that $P_{n - 1} F$ is $G$-contractible.

This tells us that $F$ is homogeneous polynomial of degree $n$, and therefore is a good candidate for the building blocks of our $G$-equivariant splitting theorem.

\chapter{Algebraic Loop Spaces}

\section{The Unitary Loop Group}

The main motivation for developing the above equivariant functor calculus was to assist us in studying the loop group $\Omega \U(V)$ as well as $\Omega \U(V; W)$, where $\U(V; W)$ is the unitary Stiefel manifold $\U(V \oplus W) / \U(W)$ for complex vector spaces $V, W$. However, this machinery does little for us directly. Instead, we seek to replace these objects with $G$-equivalent algebraic models that are easier to work with in this framework. In particular, we will develop a model for $\Omega \U(V)$ as an infinite dimensional Grassmannian. We should note that throughout this paper, because we would like to take into account the differentiable structure of the smooth manifold $Z$ when considering its loop space, we will take $\Omega Z$ to mean the space of smooth based maps $S^1 \to Z$.

First, we will need to say about about $\Omega \U(V)$ as a $G$-space. We will take the $G$-action here to be the one that acts on both $S^1$ and $\U(V)$ via complex conjugation. That is, if $\gamma \in \Omega \U(V; W)$, $z \in S^1$ and $\zeta \in G$ is the non-unit element, then we have $\zeta(\gamma)(z) := \gamma(z^*)^*$. It is clear that $\U(W) \to \U(V \oplus W) \to \U(V; W)$ is a $G$-fibration sequence, from which it follows that $\Omega \U(W) \to \Omega \U(V \oplus W) \to \Omega \U(V; W)$ is a $G$-fibration sequence. We have the following lemma that will tell us more about the fixed points of these loop spaces

\begin{lem}
Let $X$ be any smooth $G$-manifold, and consider $\Omega X$ to be a $G$-space with the above action. Then, the fixed point space $(\Omega X)^G$ is given by the space of paths $\gamma: I \to X$ such that $\gamma(0) = *$ and $\gamma(1) \in X^G$. We shall denote this space as $\Omega (X, X^G)$.
\end{lem}

\begin{proof}
Let $\gamma \in (\Omega X)^G$. This means that $g(\gamma)(z) = g \circ \gamma (gz) = \gamma(z)$, so $\overline{\gamma(\overline z)} = \gamma(z)$. This means that $\gamma$ is completely determined by its values on the upper half of $S^1$, and $\overline{gamma(-1)} = \gamma(-1)$. Thus, the elements of $(\Omega X)^G$ are in bijection with paths in $X$ that start at the base point and end at an element of $X^G$, which is precisely the space $\Omega (X, X^G)$.
\end{proof}

\begin{cor}
The fixed point space $(\Omega \U(V))^G$ is the path space $\Omega (\U(V),  \Or(V_\RR))$, where $V_\RR = V^G$, the trivial representation $W$ such that $V = W_\CC$. Furthermore, the fixed point space $(\Omega \U(V; W))^G$ is the path space $\Omega (\U(V; W), \Or(V_\RR; W_\RR))$, where
$$\Or(K; L) := \Or(K \oplus L) / \Or(L)$$
for real vector spaces $K, L$.
\end{cor}

\begin{proof}
It is clear that $\U(V)^G = \Or(V_\RR)$, and similarly, $\U(V; W)^G = \Or(V_RR; W_\RR)$. The result follows.
\end{proof}

When $\U(V; W)$ has the additional structure of a group (with $\Or(V_\RR; W_\RR)$ as a subgroup), as is the case when $W = 0$ or $W = \CC$, then we may write this path space in another form. We have the following corollary about how to do this.

\begin{cor}
When $W$ has dimension $0$ or $1$, then we have $$(\Omega \U(V; W))^G = \Omega (\U(V; W) / \Or(V_\RR; W_\RR)).$$
\end{cor}

\begin{proof}
By definition, the space $\Omega (\U(V; W), \Or(V_\RR; W_\RR))$ is the homotopy fiber of the inclusion $\Or(V_\RR; W_\RR) \hookrightarrow \U(V; W)$.

When $W = 0$, this is simply the inclusion $\Or(V_\RR) \hookrightarrow \U(V)$, which is part of the fibration sequence
$$\Or(V_RR) \to \U(V) \to \U(V) / \Or(V_\RR).$$
Extending this sequence to the left gives us the fibration sequence
$$\Omega (\U(V) / \Or(V_\RR)) \to \Or(V_\RR) \to \U(V),$$
so we see that we have $\Omega (\U(V), \Or(V_\RR)) = \Omega (\U(V) / \Or(V_\RR))$.

When $\dim W = 1$, we have $W \cong \CC$. In this case, we may make the identification $\U(V; \CC) = \SU(V \oplus \CC)$ preserving the subspace of real points $\Or(V_\RR; \RR) = \SO(V_\RR \oplus \RR)$. Rewriting $V \oplus \CC$ as $V'$, we have a fibration sequence
$$\SO(V'_\RR) \to \SU(V') \to \SU(V') / \SO(V'_\RR)$$
as before. Extending this to the left gives us the fibration sequence
$$\Omega (\SU(V') / \SO(V'_\RR)) \to \SO(V'_\RR) \to \SU(V'),$$
so we make the identification $\Omega (\SU(V'), \SO(V'_\RR)) = \Omega (\SU(V') / \SO(V'_\RR))$.

In both cases, we have $(\Omega \U(V: W))^G = \Omega ( \U(V; W) / \Or(V_\RR; W_\RR) )$.
\end{proof}

We now recall the matrix polar decomposition; we can write any matrix $A \in \CC^{n \times n}$ as $U R$, where $U \in \U_n$ and $R$ is a positive semi-definite Hermitian matrix. If $A$ is invertible, that is $A \in \GL_n(\CC)$, then $R$ is positive definite, and therefore the decomposition is unique (if we have two decompositions $U R = U' R'$, then $U^{-1} U' = R R'^{-1}$, so both sides must be the identity). This gives us a $G$-homeomorphism $\GL_n(\CC) \cong \U_n \times \R_n$, where $\R_n$ is the set of $n \times n$ positive definite Hermitian matrices, because complex conjugation distributes across products.

\begin{lem}
The space of positive definite Hermitian matrices is $G$-contractible.
\end{lem}

\begin{proof}
The Lie algebra $\fr_n$ is the vector space of all $n \times n$ Hermitian matrices. We know that every Hermitian matrix is diagonalizable by unitary matrices, so if $A \in \fr_n$, $\exists U \in \U_n$ such that $U^{-1} A U = D$ is diagonal. Furthermore, because $D$ is also Hermitian, all of its entries will be real. We can now consider the exponential map $\exp: \fr_n \to \R_n$. This is clearly a $G$-homeomorphism when restricted to the diagonal elements, as it is simply $n$ copies of the homeomorphism $\exp: \RR \to \RR_{\ge 0}$ in that case (and every point is a $G$-fixed point).

More generally, we have $\exp(A) = \exp(U D U^{-1}) = U \exp(D) U^{-1}$. Now, if $R \in \R_n$, then we can write $R = U D U^{-1}$. We know that $\exists! D'$ for which $\exp(D') = D$, so $\exp(U D' U^{-1}) = R$, so $\exp$ is surjective. Furthermore, $\exp(A)$ is diagonal iff $A$ is diagonal, because $A, \exp(A)$ have the same eigenspaces, so $\exp$ is a bijection, and therefore is a homeomorphism.

The real points of $\fr_n$ are the $n \times n$ real symmetric matrices, which are homeomorphic to the real points of $\R_n$, the $n \times n$ positive definite symmetric matrices, under the exponential map by exactly the same argument (if $A \in \fr_n$ is real, then we may diagonalize it with orthogonal matrices). Thus, $\exp$ is a $G$-homeomorphism, and therefore $\R_n$ is $G$-contractible.
\end{proof}

We may conclude from this that we have a $G$-equivalence $\GL_n(\CC) \simeq \U_n$. This allows us to replace $\Omega \U_n$ with $\Omega \GL_n(\CC)$ for the purpose of constructing the Grassmannian model.

\section{The Space of Polynomial Loops}

One well known technique in the study of loop groups, as discussed in \cite{Pressley} is to replace a given smooth loop space by a homotopy equivalent space which is easier to work with (and often algebraically defined), such as the space of polynomial loops. In our case of interest, the space $\Omega \U(V; W)$ is not necessarily a group (unless $W \cong \CC$), but we may still use a similar procedure, by replacing $\Omega \U(V \oplus W)$, and considering its image in $\Omega U(V; W)$. In this section, we shall discuss this procedure in more detail.

As alluded to already, we wish to replace the space of smooth loops $\Omega \GL_n(\CC)$ with a more algebraic object.

\begin{defn}
Let $\Omega_{pol} \GL_n(\CC)$ be the subset of $\gamma \in \Omega \GL_n(\CC)$ such that the entries of $\gamma(z)$ are polynomials in $z, z^{-1}$. We can also define $\Omega_{pol} \U_n := \Omega_{pol} \GL_n(\CC) \cap \Omega \U_n$, and more generally $\Omega_{pol} K := \Omega_{pol} \GL_n(\CC) \cap \Omega K$, for any subgroup $K \subset \GL_n(\CC)$. We point out here that the requirement of a polynomial loop $\gamma$ being in $\Omega_{pol} \U_n$ can also be expressed as $\gamma(z) \in \U_n$ for $z \in S^1$ (as this cannot be the case for all $z \in C^\times$ if $\gamma$ is nonconstant).
\end{defn}

While we are primarily interested in based loop spaces, it will be convenient to also consider the space $L X$ of unbased smooth maps $S^1 \to X$. For any Lie group $K$, we have the short exact sequence of topological groups $0 \to \Omega K \to L K \to K \to 0$, where $K$ acts on $\Omega K$ via conjugation. This sequence splits in the category of smooth manifolds when we forget the algebraic structure. In particular, this means that we have a homotopy equivalence $LK / K \to \Omega K$.

We will define $L_{pol} K$ similarly with respect to $\Omega_{pol} K$, and we shall let $L_{pol}^{(d)}$ denote the subspace of polynomial loops which contain powers of $z, z^{-1}$ of at most $d$.

We now present the following lemma from \cite{Bauer} which will be incredibly helpful for us.

\begin{lem}
Let $X$ be a compact Hausdorff space, $Y \subseteq X$ a closed subspace, and $V$ a real vector space of dimension at least $2$. Suppose that $f: X \to L (V - \{0\})$ is a continuous map such that $f(Y) \subseteq L_{pol}^{(c)} (V - \{0\})$ for some $c \ge 0$.
Then, there is an integer $d > 0$ and a continuous map $g: X \to L_{pol}^{(d)} (V - \{0\})$ such that $f_t(x) := (1 - t) f(x) + t g(x)$ defines a loop in $V - \{0\}$ for all $t \in I, x \in X$ and $g(y) = f(y)$ for all $y \in Y$. In other words, $f$ is homotopic to $g$ via $f_t$, which is constant on the entirety of $Y$.
\end{lem}

This tells us how to deform any loop to an algebraic loop, and therefore gives us the following 

\begin{cor}
If $V$ is a complex vector space and $\dim_\CC V \ge 2$, then the inclusion
$$\Omega_{pol} (V - \{0\}) \hookrightarrow \Omega (V - \{0\})$$
is a $G$-equivalence.
\end{cor}

\begin{proof}
We can apply the lemma to both $V - \{ 0 \}$ and $V_\RR - \{ 0 \}$, after noting the $G$-fixed points of both spaces, giving us homotopy equivalences $L_{pol} (V - \{0\}) \hookrightarrow L (V - \{0\})$, for $W = V, V_\RR$, and therefore $\Omega_{pol} (W - \{0\}) \to \Omega (W - \{0\})$.

Now, the $G$-fixed points of $\Omega (V - \{0\})$ are precisely $\Omega ((V - \{0\}), (V_\RR - \{0\}))$, which is $\Omega ((V - \{0\}) / (V_\RR - \{0\}))$.

Similarly, for the polynomial loops, the $G$-fixed points are the set of $\gamma \in \Omega_{pol} (V - \{0\})$ such that $\gamma(z)$ has only real coefficients, which is similarly the quotient of $\Omega_{pol} (V - \{0\})$ by the subgroup $\Omega_{pol} (V_\RR - \{0\})$. By comparing the fiber sequences, the result follows.
\end{proof}

We will use this to prove the following proposition.

\begin{prop}
The inclusion $\Omega_{pol} \GL_n(\CC) \hookrightarrow \Omega \GL_n(\CC)$ is a $G$-equivalence.
\end{prop}

\begin{proof}
There are many known proofs of the non-equivariant version of this, such as Theorem 1.4 from \cite{Mitchell}, and most famously Theorem 8.6.6 from \cite{Pressley}. While the latter proof could be adapted to the equivariant statement, a new method will be presented below, as it better highlights why one should expect the result to be $\ZZ / 2$-equivariant.

For any $\RR$-algebra $A$, we can think of $\GL_n(A)$ as a fiber bundle over $A^n - \{0\}$,
with the fiber over $v$ being the space of all $M \in \GL_n(A)$ such that $M e_n = v$. Such $M$ are in natural bijection with $\GL(A^n / \langle v \rangle) \times \hom_A(A^n / \langle v \rangle, \langle v \rangle)$.
That is, we have the fiber sequence
$$\GL_{n - 1}(A) \times A^{n - 1} \to \GL_n(A) \to A^n - \{0\}.$$

By setting $A = \RR, \CC$, we get a $G$-fiber sequence
$$\GL_{n - 1}(\CC) \times \CC^{n - 1} \to \GL_n(\CC) \to \CC^n - \{0\},$$
and therefore a $G$-fiber sequence
$$\Omega (\GL_{n - 1}(\CC) \times \CC^{n - 1}) \to \Omega \GL_n(\CC) \to \Omega (\CC^n - \{0\}),$$
where we take $\Omega( - )$ to mean $\Map_*(S^\sigma, -)$, for the sign representation $\sigma$.

Setting $A = \RR[z, z^{-1}], \CC[z, z^{-1}]$ gives us the $G$-fiber sequence
$$\Omega_{pol} (\GL_{n - 1}(\CC) \times \CC^{n - 1}) \to \Omega_{pol} \GL_n(\CC) \to \Omega_{pol} (\CC^n - \{0\}).$$

We now have the map of fibrations
$$\begin{array}{ccccc}
\Omega_{pol} (\GL_{n - 1}(\CC) \times \CC^{n - 1}) & \to & \Omega_{pol} \GL_n(\CC) & \to & \Omega_{pol} (\CC^n - \{0\}) \\
\downarrow & & \downarrow & & \downarrow \\
\Omega (\GL_{n - 1}(\CC) \times \CC^{n - 1}) & \to & \Omega \GL_n(\CC) & \to & \Omega (\CC^n - \{0\})
\end{array}$$

We know from the previous lemma that the right vertical arrow is a $G$-equivalence. It is clear that both $\Omega_{pol}$ and $\Omega$ distribute over products, and yield $G$-contractible spaces when applied to a complex vector space. Thus, the left vertical arrow is $G$-equivalent to $\Omega_{pol} \GL_{n - 1}(\CC) \to \Omega \GL_n(\CC)$, so the result will follow from induction and the five lemma if we can establish the base case of $n = 1$. We note that it is sufficient to prove this for the free loop map $L_{pol} \GL_{n - 1}(\CC) \to L \GL_n(\CC)$.

We know that $\GL_1(\CC) = \CC^\times$, so our map of interest is $\CC^\times(\CC[z, z^{-1}]) \to L \CC^\times$. The coordinate ring of $\CC^\times$ is naturally $\CC[z, z^{-1}]$, so the domain of this map is naturally the space of $\CC$-algebra endomorphisms of $\CC[z, z^{-1}]$. Because such a map $\varphi$ is completely determined by the image of $z$, and $\varphi(z)$ must be invertible, we have $\varphi(z) = \alpha z^k$, for $\alpha \in \CC, k \in \ZZ$. This means that $L_{pol} \CC^\times \simeq \CC^\times \times \ZZ$.

We know that $\Omega \CC^\times \simeq \ZZ$, so $L \CC^\times \simeq \CC^\times \times \ZZ$. It is also clear that the inclusion is a $G$-equivalence, as it is on both parts of the product- the loop given by $\gamma(z) = z^k$ lies in the $k$ component of both the domain and codomain, and the $\CC^\times$ component simply tells us the base point which must be preserved as well.

Thus, we may conclude that $\Omega_{pol} \GL_n(\CC) \to \Omega \GL_n(\CC)$ is a $G$-equivalence.
\end{proof}

We have the following important corollary.

\begin{cor}
The inclusion $\Omega_{pol} \U_n \hookrightarrow \Omega \U_n$ is a $G$-equivalence.
\end{cor}

\section{Algebraic Loops of the Homogeneous Space}

It turns out that $\Omega \U(V; W)$ is in many ways simpler when $W \ne 0$, despite failing to be a group when $\dim(W) > 1$. In particular, we have the following lemma.

\begin{lem}
Let $V, W \in \cJ$, with $W \ne 0$. Then, $\Omega \U(V; W)$ is $G$-connected.
\end{lem}

\begin{proof}
We know that $\pi_0^1(\Omega \U(V; W)) \cong \pi_1(\U(V; W))$, which is the cokernel of the map $\pi_1(\U(W)) \to \pi_1(\U(V \oplus W))$ (since $\pi_0(\U(W)) = 0$). Since this map is an isomorphism, we have that $\pi_0^1(\Omega \U(V; W)) = 0$.

We know from before that $(\Omega \U(V; W))^G = \Omega (\U(V; W), \Or(V_\RR; W_\RR))$, so
$$\pi_0^G(\Omega \U(V; W)) = \pi_0( \Omega (\U(V; W), \Or(V_\RR; W_\RR)) ),$$
which is trivial, as it is the homotopy fiber of the inclusion $\Or(V_\RR; W_\RR) \hookrightarrow \U(V; W)$, both of which are connected and the latter of which is simply connected.
\end{proof}

In order to further study $\Omega \U(V; W)$, we seek a $G$-equivalent algebraic model, in the same manner as above. The natural candidate for this is the orbit space
$$\Omega_{pol} \U(V; W) := \Omega_{pol} \U(V \oplus W) / \Omega_{pol} \U(W),$$
which can also be realized as the image of $\Omega_{pol} \U(V \oplus W)$ under the defining quotient map $U(V \oplus W) \to U(V; W)$.

It is clear that $\Omega_{pol} \U(W) \to \Omega_{pol} \U(V \oplus W) \to \Omega_{pol} \U(V; W)$ is a fibration sequence on the underlying space. Furthermore, the sequence of fixed points is simply that arising from the subgroups of matrices having only real coefficients, which again yields a fibration sequence. Thus, we have a $G$-fibration sequence.

This, along with the maps $\Omega_{pol} \U(V) \to \Omega \U(V)$ gives us the following diagram of $G$-fibration sequences:
$$\begin{array}{ccccc}
\Omega_{pol}(W) & \to & \Omega_{pol} \U(V \oplus W) & \to & \Omega_{pol} \U(V; W) \\
\downarrow{\simeq} & & \downarrow{\simeq} & & \downarrow \\
\Omega \U(W) & \to & \Omega \U(V \oplus W) & \to & \Omega \U(V; W)
\end{array}$$

We know from above that the first two vertical arrows are $G$-homotopy equivalences. Thus, for $H \le G$, $n \ge 1$, we have the following diagram from the long exact homotopy sequences: 
$$\begin{array}{ccccccccc}
\pi_n^H( \Omega_{pol} \U(W) ) & \to & \pi_n^H( \Omega_{pol} \U(V \oplus W) ) & \to & \pi_n^H( \Omega_{pol} \U(V; W) ) & \to & \pi_{n - 1}^H( \Omega_{pol} \U(W) ) \\
\downarrow{\cong} & & \downarrow{\cong} & & \downarrow & & \downarrow{\cong} \\
\pi_n^H( \Omega \U(W) ) & \to & \pi_n^H( \Omega \U(V \oplus W) ) & \to & \pi_n^H( \Omega \U(V; W) ) & \to & \pi_{n - 1}^H( \Omega \U(W) ) \\
\downarrow{\cong} & & \downarrow{\cong} & & \downarrow{\cong} & & \downarrow{\cong} \\
\pi_{n + 1}^H( \U(W) ) & \to & \pi_{n + 1}^H( \U(V \oplus W) ) & \to & \pi_{n + 1}^H( \U(V; W) ) & \to & \pi_n^H( \U(W) ) \\
\end{array}$$

We may employ the five lemma to conclude that $\pi_n^H(\Omega_{pol} \U(V; W)) \to \pi_n^H(\Omega \U(V; W))$ is an isomorphism for $H = 1, G$.

It remains only to show that $\pi_0^H( \Omega_{pol} \U(V; W) ) \to \pi_0^H( \Omega \U(V; W) ) \cong \pi_1^H( \U(V; W) )$ is an isomorphism for $H = 1, G$. We already know from the above lemma that $\Omega \U(V; W)$ is $G$-connected when $W \ne 0$, so we need only verify that $\Omega_{pol} \U(V; W)$ is $G$-connected as well. However, . We may conclude that $\Omega_{pol} \U(V; W) \simeq \Omega \U(V; W)$, so we may replace the loop space of the Stiefel manifold with its algebraic analogue.

\section{A Grassmannian Model for the Loop Group}

One of the advantages working with algebraic loops is that we can describe our space as an infinite dimensional Grassmanian. In this section, we will construct algebraic varieties $S_k(V) \subset \Omega_k \U(V)$, as detailed in \cite{Crabb}, where $\Omega_k \U(V) = \{ \gamma \in \U(V) \ | \ \deg( \det(\gamma) ) = k \}$, and use them to construct a Grassmannian model for $\Omega \U(V)$.

We consider the subclass of algebraic loops given by $p_E(z) := z \pi_U \oplus 1 \pi_{U^\perp}$ for some linear subspace $E \subseteq V$. We take $S_k(V)$ to be the image of the map $p: \PP(V) \times \cdots \times \PP(V) \to \Omega \U(V)$ given by $(L_1, \ldots, L_k) \mapsto p_{L_1} \cdots p_{L_k}$. Note that $S_0(V) = \{ 1 \}$ and $S_1(V)$ can be identified with $\PP(V)$. Furthermore, if $k < \dim(V)$, for a generic set of $k$ lines in $V$, none of the $L_i$ will lie in any plane defined by any subset of the others, and therefore the lines $L_1, \ldots, L_k$ will define a $k$-plane $E$. In this case, the corresponding element of $S_k(V)$ will be precisely $p_E$, which can be identified with $E \in \Gr_k(V)$. In other words, when $k < \dim(V)$, a generic point of the domain will yield an element of $\Gr_k(V)$, suggesting that $S_k(V)$ is naturally some sort of extension of the usual complex Grassmannian.

Proposition 2.3 in \cite{Crabb} tells us the following.

\begin{prop} The union $\displaystyle X(V) := \bigsqcup_{k \ge 0} S_k(V)$ consists of those $\gamma \in \Omega \U(V)$ for which we can write $\displaystyle \gamma(z) = \sum_{i \ge 0} \gamma_i z^i$, for $\gamma_i \in \End(V)$. Furthermore, any such $\gamma$ has a canonical (but not necessarily unique) expression as a product $p_{E_1} \circ \cdots \circ p_{E_r}$, where $E_1, \ldots, E_r$ are subspaces of $V$.
\end{prop}

The essence of the proof is iteratively factoring $\gamma$ from the right, at each step taking $\gamma' = \gamma \circ P_E^{-1}$, for the maximal such $E$ for which $\gamma'$ contains no $z^{-1}$ term. This happens if and only if $E \subseteq \ker \gamma_0$. It is easy to see that taking $E = \ker \gamma_0$ gives the maximal such $E$, and that $\dim \ker \gamma'_0 < \dim \ker \gamma$. Thus, we will eventually have $\ker \gamma_0 = 0$, meaning that $\gamma_0$ is invertible, and therefore $\gamma = 1$.

We will now describe how to extend the tautological bundle $\tau$ on $\Gr_k(V) \subset S_k(V)$ to the entirety of $S_k(V)$, and in doing so give a description of $S_k(V)$ as a complex algebraic variety. Let $\tilde H(V) = V \otimes_\CC \CC[z, z^{-1}]$, $H(V)$ the subspace $V \otimes_\CC \CC[z]$, and let $I(V) := \tilde H(V) / H(V)$, which we can identify with $\bigoplus_{i > 0} z^{-i} V$. We have a filtration of the $\CC[z]$-module $I(V)$ by submodules $I^k V = \{ m \in I(V) \ | \ z^k m = 0 \}$, or $\bigoplus_{i = 1}^k z^{-i} V$, using the second description.

It is clear that we have an action of $X(V)$ on $\tilde H(V)$ that carries $H(V)$ to $H(V)$, and therefore induces an action on $I(V) \to I(V)$. Because every $\gamma \in X(V)$ is invertible, the induced map $\gamma_*: I(V) \to I(V)$ is surjective, and furthermore that each $\gamma \in X(V)$ gives a unique map on $I(V)$. It is also clear that if $\gamma \in S_k(V)$, then $\dim \ker \gamma_* = k$ and $\ker \gamma_* \subseteq I^k V$. This allows us to define a map $\xi: S_k(V) \to \Gr_k(I^k V)$ by $\gamma \mapsto \ker \gamma_*$.

This is important for two reasons; the first of which is that it gives us a $G$-equivariant embedding of $S_k(V)$ into the Grassmannian $\Gr_k(I^k V)$. (Note that if $\ker \gamma_* = \ker \gamma'_*$, then $\gamma_*' = \alpha_* \gamma_*$ for some invertible $\alpha_*$, meaning $\alpha \in S_0(V)$. This means that $\alpha = 1$, and therefore $\gamma = \gamma'$, so indeed we have an embedding.) In fact, by identifying $I^1 V$ with $V$, and noticing that the image of all of the $p_E$, for $E \subseteq V$ of dimension $k$ is precisely $\Gr_k(I^V)$, we obtain $\Gr_k(V) \subseteq S_k(V) \subseteq \Gr_k(I^k V)$.

Now, we claim that the image of $S_k(V)$ under $\xi$ is $\cM_k(V) := \{ M \in \Gr_k(I^k V) \ | \ zM \subset M \}$. It is clear that each $\xi(\gamma)$ has this property, so we must show the converse; that every $M \in cM$ is the image of some $\gamma \in S_k(V)$. We will proceed by induction. Suppose that we know that this is true for $\cM_i(V)$ for $0 \le i \le k - 1$, and let $M \in \cM_k(V)$. Because $z^k M = 0$, we must have $z M \ne M$, so $z M \in \Gr_i(I^i V)$, for $i < k$. This means that we have some $\gamma' \in S_i(V)$ such that $\xi(\gamma') = zM$. This means that $M = z^{-1} \xi(\gamma') \oplus \ker(z: M \to M)$. Since $\ker(z |_M) \subset I^1 V$, we can identify it with some subspace $E \subset V$ of dimension $k - i$, and therefore $\xi(p_E) = \ker(z |_M)$. We now have $M = z^{-1} \xi(\gamma') \oplus \xi(p_E) = \xi(\gamma' \circ p_E)$, so $M \in \im(\xi |_{S_k(V)})$. Because we have the trivial base case of $k = 0$ we may conclude that $\xi: S_k(V) \to \cM_k(V)$ is a $G$-homeomorphism.

This tells us that $X(V)$ can be identified with the infinite dimensional Grassmannian $\cM(V) := \bigsqcup_{k \ge 0} \cM_k(V)$, the space of all finite dimensional subspaces of $I(V)$ which are closed under multiplication by $z$. For those familiar with the affine Grassmanian $\cG(V)$, we will note that $X(V)$ is a proper subspace, with nonzero intersection with only those connected components indexed by nonnegative integers. Furthermore, the inclusion $S_k(V) \to \cG_k(V)$ becomes an equivalence up to a dimension that increases to infinity as $k \to \infty$.

\section{The Grassmanian Model for $\Omega \U(V; W)$}
As we have throughout this chapter, we will take $V, W$ to be finite dimensional complex vector spaces. Throughout this section, we will largely follow \cite{Crabb}, and as such, will define $S^k(V; W) := \iota^*(S_k(V; W)) \subset \Omega \U(V; W)$ where $\iota^*: \Omega \U(V \oplus W) \to \Omega \U(V; W)$ comes from the trivial inclusion $\iota: V \to V \oplus W$ (which therefore induces the defining quotient map $\U(V \oplus W) \to \U(V; W)$). In practice, we can think of $S^k(V; W)$ as the quotient of $S_k(V \oplus W)$ on the right by $S_1(W)$ (and therefore $S_i(W)$, for all $1 \le i \le k$).

We define $X(V; W) := \bigcup_{k \ge 0} S^k(V; W)$, which we will later show is $G$-equivalent to the entire space $\Omega \U(V; W)$. This is actually the same filtration of a homotopy equivalent replacement for $\Omega \U(V; W)$ that Mitchell used in his initial work on the subject.

While the varieties $S_k(V)$ are all disjoint in $\Omega \U(V)$, $\{ S^k(V; W) \}$ give us an honest filtration of $X(V; W)$ as $\iota^*( S_{k - 1}(V \oplus W) ) = \iota^*( S_{k - 1}(V \oplus W) \circ S_1(W) ) \subset \iota^*( S_k( V \oplus W ) )$. We would like to know the associated graded pieces of the filtration, so we would like to get a handle on $S^k(V; W) - S^{k - 1}(V; W)$. We have the following two lemmas from \cite{Crabb} that will help us do this.

\begin{lem}
We can write every element of $S_k(V \oplus W)$ as a product $g \circ h$, where $g \in S_\ell (V \oplus W), h \in S_{k - \ell}(V \oplus W)$, and $g$ is not divisible from the right by any element of $S_1(W)$.
\end{lem}

\begin{proof}
Let $\gamma \in S_k(V \oplus W)$, and let $E = \ker \gamma \cap W$. If $E = 0$, we are done. Otherwise, $p_E$ divides $\gamma$ from the right, so let $\gamma' = \gamma \circ p_E^{-1}$. We then have $\gamma = \gamma' \circ p_E$. We now repeat this procedure, replacing $\gamma$ with $\gamma'$. Because we reduce the dimension of $\ker \gamma$ by at least $1$ each time, after $r$ steps for some $r \le k$, we have the factorization $\gamma = \tilde \gamma \circ p$, where $\ker \tilde \gamma \cap W = 0$ and $p = p_{E_r} \circ \cdots \circ p_{E_1}$, for subspaces $E_1, \ldots, E_r \subseteq W$. (Note that as written, this produces the canonical factorization of $p \in S_\ell(W)$.)
\end{proof}

\begin{lem}
The map $\iota^*$ restricts to a $G$-homeomorphism
$$S_k(V \oplus W) - S_{k - 1}(V \oplus W) \circ S_1(W) \to S^k(V; W) - S^{k - 1}(V; W).$$
\end{lem}

\begin{proof}
It is clear that the kernel of $\iota^*: X(V \oplus W) \to X(V; W)$ is $X(W)$. Thus, the restricted map
$\iota^* |_{S_k(V \oplus W) - S_{k - 1}(V \oplus W) \circ S_1(W)}$ is injective, and is therefore an isomorphism.
\end{proof}

By our factorization lemma, we can identify $S_k(V \oplus W) - S_{k - 1}(V \oplus W) \circ S_1(W)$ with $\cN_k(V; W) = \{ N \in S_k(V \oplus W) \ | \ N \cap I(W) = 0 \}$, and therefore by the second lemma, we can do the same for $S^k(V; W) - S^{k - 1}(V; W)$. We can also see that $\cN_k$ can be identified with the total space of a vector bundle. It is clear that we have a projection map $\cN_k(V; W) \to S_k(V)$, as for any $N \in cN_k(V; W)$, the restriction of the projection map $\pi: I(V \oplus W) \to I(V)$ to $N$ is nondegenerate. Thus, we obtain a surjective map $\pi_*: \cN_k(V; W) \to S_k(V)$, for which the fiber above $M \in S_k(V)$ is the set of all linear maps $f: M \to I(W)$ such $f (z m) = z f(m)$, by closure under multiplication by $z$.
Such maps necessarily have the form
$$f(m) = \sum_{i > 0} z^{-i} e ( z^{i - 1} m ),$$
for all $m \in M$, where $e: M \to W$ is some linear map. Thus, the fiber over $M \in S_k(V)$ is $\hom(M, W)$, so we can express the total space as $\cN_k(V; W) = \hom(\tau, W)$, where $\tau$ is the tautological bundle over $S_k(V)$.

This allows us to $G$-equivariantly identify the associated graded pieces of the filtration as Thom spaces $S^k(V; W) / S^{k - 1}(V; W) \cong S_k(V)^{\hom(\tau, W)}$. We would like to make two more notes here, the first being that because $S_k(V) \subset \Gr_k(I^k V)$, the latter of which is the (homotopy) orbit space of the $\U_k \rtimes G$-space $\mor(\CC^k, I^k V)$. We can form the pullback diagram
$$\begin{array}{ccc}
\tilde S_k(V) & \to & \mor(\CC^k, I^k V) \\
\downarrow & & \downarrow \\
S_k(V) & \to & \Gr_k(I^k V)
\end{array}$$

This gives us a natural $G$-equivalence $S_k(V) \simeq (\tilde S_k(V))_{h \U_k}$, and therefore a natural $G$-equivalence $S_k(V)^{\hom(\tau, W)} \simeq ( \tilde S_k(V)_+ \wedge S^{k W} )_{h \U_k}$.
It remains to show the following proposition, the non-equivariant version of which is used in \cite{Crabb} and previously by Mitchell and Richter.

\begin{prop}
 When $W \ne 0$, the inclusion $X(V; W) \to \Omega \U(V; W)$ is a $G$-equivalence.
\end{prop}

\begin{proof}
It is clear that this inclusion factors as $X(V; W) \to \Omega_{pol} \U(V; W) \to \Omega \U(V;W)$, so we need only show that the former map is a $G$-equivalence.

Suppose $\dim V = n, \dim W = m$. From 2.23 in \cite{Crabb}, we have the filtration
$$S_k(V) \subseteq z^{-1} S_{k + n}(V) \subseteq \cdots \subseteq z^{-r} S_{k + rn}(V) \subseteq \subseteq \Omega_k \U(V),$$
for every $k \ge 0$, where $z = z I \in S_n(V)$ is the central element.

It is clear that the union $\bigcup_{r \ge 0} z^{-r} S_{k + rn}(V)$ is all of $\Omega_{pol} \U(V) \cap \Omega_k \U(V)$. We also know that $\Omega_{pol} \U(V \oplus W) \cap \Omega_k \U(V \oplus W)$ surjects onto $\Omega_{pol} \U(V; W)$ for any $k \ge 0$, so we may conclude the same for $\displaystyle \bigcup_{r \ge 0} z^{-r} S_{k + r(n + m)}(V \oplus W)$, where $z = z I \in S_{n + m}$ here. In particular, the union over all $k$ surjects onto $\Omega_{pol} \U(V; W)$.

Now, if we take $s = 1 I_V \oplus z I_W$, we obtain the similar filtration
$$S_k(V \oplus W) \subseteq s^{-1} S_{k + m}(V \oplus W) \subseteq \cdots \subseteq s^{-r} S_{k + rm}(V \oplus W) \subseteq \subseteq \Omega_k \U(V \oplus W).$$

By construction, the image of $s^{-r} S_{k + rm}(V \oplus W)$ in $\Omega_{pol} \U(V; W)$ ignores the factors of $s$, and therefore is exactly $S^{k + rm}(V \oplus W)$, and therefore for any $k$ the union
$$\bigcup_{r \ge 0} s^{-r} S_{k + rm}(V \oplus W)$$
has image equal to $X(V; W)$ (and furthermore the union of this over all $k$ has image $X(V; W)$). Therefore, if we can show for sufficiently large $k$ that the inclusion
$$\bigcup_{r \ge 0} s^{-r} S_{k + rm}(V \oplus W) \hookrightarrow \bigcup_{r \ge 0} z^{-r} S_{k + r(n + m)}(V \oplus W)$$
is a $G$-equivalence, then we are done.

However, the inclusion $s^{-r} S_{k + rm}(V \oplus W) \hookrightarrow z^{-r} S_{k + r(n + m)}(V \oplus W)$ induces a map in integral homology of the underlying spaces
$$b_0^{-rm} \Sym^{k + rm} H_*(\CP(V \oplus W); \ZZ) \to b_0^{-r(m + n)} \Sym^{k + r(m + n)} H_*(\CP(V \oplus W); \ZZ),$$
which is evidently an isomorphism up to some degree that increases to infinity with $k$. Thus, the inclusion is a homotopy equivalence for the underlying spaces as $k \to \infty$.

As for the fixed points, we have $s^{-r} S^\RR_{k + rm}(V_\RR \oplus W_\RR) \hookrightarrow z^{-r} S^\RR_{k + r(n + m)}(V_\RR \oplus W_\RR)$, which induces a map in mod $2$ homology
$$b_0^{-rm} \Sym^{k + rm} H_*(\RP(V \oplus W); \ZZ / 2) \to b_0^{-r(m + n)} \Sym^{k + r(m + n)} H_*(\RP(V \oplus W); \ZZ / 2),$$
which similarly is an isomorphism up to some degree that increases to infinity with $k$. We conclude that the inclusion of fixed points is also a homotopy equivalence as $k \to \infty$, and therefore that the inclusion is a $G$-equivalence as $k \to \infty$.

We may therefore conclude that $X(V; W)$ is $G$-equivalent to $\Omega \U(V; W)$.

\end{proof}

We now have a filtration of an algebraic replacement for $\Omega U(V; W)$ for which the associated graded components are $G$-homotopy equivalent to the precursors (before applying $Q$) of homogeneous polynomial functors for the input $W$.

\chapter{The Main Splitting Theorem}

\section{The General Equivariant Splitting}

We are now ready to establish the equivariant analogue of the splitting theorem from \cite{Arone}. 

\begin{thm}
Let $F: \cJ \to G-\Top$ be a functor with a $G$-filtration $F_0 \subset F_1 \subset \cdots $ such that\\
1) $F_0 = *$ \\
2) The action of $G$ on $F_{n - 1}(V)$ is that inherited from considering $F_{n - 1}(V)$ as a subspace of $F_n(V)$ for all $V \in \cJ$ \\
3) The homotopy cofiber of the inclusion $F_{n - 1}(V) \hookrightarrow F_n(V)$ is $G$-equivalent to $(\Th \wedge S^{n V_\CC})_{h \U_n}$, where $\Th$ is a spectrum with a $\U_n \rtimes G$ action.

Then, we have a natural stable $G$-equivalence of functors $F \simeq \bigvee_{n > 0} (F_n / F_{n - 1})$.
\end{thm}

This will give us our desired equivariant splitting of the space $X(V_\CC; W_\CC)$, and therefore of $\Omega \U(V_\CC; W_\CC)$. As stated above, we have a $G$-equivariant filtration $S^n(V_\CC; W_\CC)$ of the space $X(V_\CC; W_\CC)$.

From \cite{Arone}, we have the following proposition.

\begin{prop}: Let $F_1 \to F_2 \to F_3$ be a $G$-fibration sequence of functors. It induces $G$-fibration sequences $T_n F_1 \to T_n F_2 \to T_n F_3$ and $P_n F_1 \to P_n F_2 \to P_n F_3$.
\end{prop}

\begin{proof} Functoriality means that we get maps between $G$-fibration sequences
$$\begin{array}{ccccc}
F_1(V) & \to & F_2(V) & \to & F_3(V) \\
\downarrow & & \downarrow & & \downarrow \\
F_1(V \oplus U) & \to & F_2(V \oplus U) & \to & F_3(V \oplus U)
\end{array}$$

Taking the homotopy limit over all $U \in \cC_n$ gives the statement for $T_n$. This in turn allows us to form the diagram
$$\begin{array}{ccccc}
F_1(V) & \to & F_2(V) & \to & F_3(V) \\
\downarrow & & \downarrow & & \downarrow \\
T_n F_1(V) & \to & T_n F_2(V) & \to & T_n F_3(V) \\
\downarrow & & \downarrow & & \downarrow \\
T_n^2 F_1(V) & \to & T_n^2 F_2(V) & \to & T_n^2 F_3(V) \\
\downarrow & & \downarrow & & \downarrow \\
\vdots & & \vdots & & \vdots
\end{array}$$

Each row of this is clearly a $G$-fibration sequence. Taking the homotopy colimit of this gives us a $G$-fibration sequence $P_n F_1(V) \to P_n F_2(V) \to P_n F_3(V)$, as claimed.
\end{proof}

Now, suppose $F: \cJ_0 \to G-\Top_*$ is a continuous functor such that there exists a filtration of $F$ by sub-functors $F_n$ such that $F_0(V) \equiv *$, and for all $n \ge 1$, $\hocofiber[F_{n-1}(V) \to F_n(V)]$ is (up to a natural weak equivalence) of the form $(X_n \wedge S^{nV})_{h \U_n}$, where $X_n$ is a based space with an action of $\U_n \rtimes G$. We claim that $Q F_n$ is polynomial of degree $\le n$.

We will in fact show that $Q \Sigma^k F_n$ is polynomial of degree $\le n$. Suppose for the sake of induction that we know that $Q \Sigma^k F_n$ is polynomial of degree $\le n$ for all $k$. Then, we can consider the fibration sequence $Q \Sigma^k F_n \to \Sigma^k F_n / F_{n - 1} \to Q \Sigma^{k + 1} F_{n - 1}$. We know that the last two functors are polynomial of degree $\le n$ because of the condition on the given filtration of $F$ and the induction hypothesis. Then, because the homotopy fiber of a natural transformation of functors which are polynomial of degree $\le n$ is also polynomial of degree $\le n$, we conclude that $Q \Sigma^k F_n$ is polynomial of degree $\le n$, thus proving our claim.

Now, for the main theorem. We can consider the fibration sequence
$$Q F_{n - 1} \to Q F_n \to Q (F_n / F_{n - 1}),$$
and the following diagram $$\begin{array}{ccccc}
Q F_{n-1} & \to & Q F_n & \to & Q (F_n / F_{n-1}) \\
\downarrow{\simeq} & & \downarrow & & \downarrow \\
P_{n - 1} Q F_{n-1} & \to & P_{n - 1} Q F_n & \to & P_{n - 1} Q (F_n / F_{n-1})
\end{array}$$

Because the bottom row is a $G$-fibration sequence and the bottom right space is $G$-contractible, the bottom left map is a weak $G$-homotopy equivalence. This means that the composition $Q F_{n - 1} \to Q F_n \to P_{n - 1} Q F_n$ is a weak $G$-homotopy equivalence, and therefore $Q F_{n - 1}$ is a $G$-homotopy retract of $Q F_n$. Because the inclusion $F_{n - 1} \to F_n$ is an infinite loop map, it follows that $\Sigma^\infty F_n \simeq \Sigma^\infty (F_{n - 1} \wedge (F_n / F_{n - 1})$, giving us the desired $G$-equivariant splitting.

In our particular case, this means that we have a $G$-equivariant stable splitting
$$X(V_\CC; W_\CC) \simeq \bigvee_{n > 0} S_n(V_\CC)^{\hom(\tau, W_\CC)},$$
and therefore of $\Omega \U (V_\CC; W_\CC)$.

\section{The Splitting of the Fixed Points}

For this section, we will take $V, W$ to be real vector spaces.
From the previous section, we know that the stable splitting of $\Omega \U(V_\CC; W_\CC)$ is $\ZZ / 2$-equivariant, so we naturally get a splitting of the fixed point spectrum
$$(\Omega \U(V_\CC; W_\CC))^G \simeq (\bigvee_{n > 0} S_n(V)^{\hom(\tau, W)})^G = \bigvee_{n > 0} (S_n(V)^{\hom(\tau, W)})^G.$$

Furthermore, we may apply the geometric fixed points functor to obtain a stable splitting of the geometric fixed points, which are the real points, in the algebaic variety sense, of $\Omega \U(V_\CC; W_\CC)$. We saw previously that these may be described by the path space $\Omega ( \U(V_\CC; W_\CC), \Or(V; W) )$.

On the other side, we have a wedge sum of terms which are Thom spaces over the algebraic varieties $S_n(V_\CC)$. Because geometric fixed points distribute over vector bundles (and therefore Thom spaces), we point out that the real points of $\Gr_n(V_\CC)$ are given by the real $n$-dimensional Grassmannian of $V$, which we shall refer to as $Gr_n^\RR(V)$, and the real points of the fiber $\hom(\CC^n, W_\CC)$ are $\hom(\RR^n, W)$. This yields the bundle $\hom(\tau, W)$, where by a slight abuse of notation we will take $\tau$ to also mean the tautological bundle over the real Grassmannian.

Although $S_n$ was defined using complex valued polynomials on $\CC^\times$, the relevant action of $\ZZ / 2$ fixes the complex variable $z$. Thus, we are able to give a relatively simple description of the real points of $S_n(V_\CC)$. Recalling that we can define this space as a complex subvariety of $\Gr_n(I^n V_\CC)$, defined as $\{ M \in \Gr_n(I^n V_\CC) \given zM \subseteq M \}$. Its geometric fixed points are clearly given by $S_n^\RR(V) := \{ M \in \Gr^\RR_n(I^n V) \given zM \subseteq M \}$. Because the (geometric) fixed points of the $G$ action are exactly the real points in the scheme theoretic sense, it is clear that the tautological bundle over $S_n^\RR(V)$ extends that of the real Grassmannian in the same way as before, and furthermore, that the summands are $S_n^\RR(V)^{\hom(\tau, W)}$.

Thus, we obtain the stable splitting of the geometric fixed points
$$\Omega ( \U(V_\CC; W_\CC), \Or(V; W) ) \simeq \bigvee_{n = 1}^\infty S_n^\RR(V)^{\hom(\tau, W)}.$$

In the special case of $W \cong \RR$, we may take $V' = V \oplus \RR$ to obtain
$$\Omega ( \SU(V'_\CC) / \SO(V') ) \simeq \bigvee_{n = 1}^\infty S_n^\RR(V)^{\hom(\tau, \RR)} \simeq \bigvee_{n = 1}^\infty S_n^\RR(V)^{\hom(\tau, \RR)}.$$

Also, when $V \cong \RR$, we have $\U(\CC; W_\CC) = \BS(W_\CC) = \Sigma S^{W_\CC'}$, if $W = W' \oplus \RR$. Furthermore, $S_n(V_\CC)$ is the single point $n V_\CC = \CC^n$. Thus, we simply recover the $G$-equivariant James splitting
$$\Omega \Sigma S^{W'_\CC} = \bigvee_{n > 1} S^{n W'}.$$

It is also clear from our constructions that the natural map arising from the fibration sequence $\Omega \U(V_\CC; W_\CC) \to \Omega ( \U(V_\CC; W_\CC), \Or(V; W) )$ stably decomposes as a wedge sum of the maps $S_n(V_\CC)^{\hom(\tau, W_\CC)} \to S_n^\RR(V)^{\hom(\tau, W)}$.

\chapter{Future Questions}

We may be a to gain some insight into the $E_2$ algebra structure of the loop group $\Omega \SU_n$, via the $G$-equivariant stable splitting
$$\Omega \SU_n \simeq \bigvee_{m > 1} S^{m \CC^{n - 1}},$$
as it is known that polynomial of functors of degree $n$ are related to the little disks operad.

It may be possible to work out a theory of equivariant symplectic calculus as well, where the group acting is the Klein four group $V_4$, with the action of each nontrivial element given by conjugating two of the three imaginary units. In this case, because $\HH - \{0\} \simeq S^3$, algebraic objects analogous to those before would model spaces of the form $\Map_*(S^3, X)$, such as $\Omega^3 \Sp_n$, rather than the loop spaces.

Based on this, I propose the following conjecture.

\begin{conj}
Let $V, W$ be finite dimensional real vector subspaces of $\cU$ as before, inheriting the standard inner product. Let $\Sp(V_\HH; W_\HH)$ be the space of q-linear transformations that preserve the induced quaternionic inner product, considered as a $V_4$-space. Then, there is a $V_4$-equivariant stable splitting of $\Sp(V_\HH; W_\HH)$ into Thom spaces of finite dimensional vector bundles with base spaces that can be expressed as algebraic subvarieities of quaternionic Grassmanians of either $V_\HH$ or finite multiples of $V_\HH$.
\end{conj}

One could also consider the motivic homotopy theory analogue of any of these splitting theorems. 

\pagebreak




\singlespacing

\nocite{*}



\thispagestyle{plain}


\let\oldthebibliography=\thebibliography
\let\endoldthebibliography=\endthebibliography
\renewenvironment{thebibliography}[1]{
  \begin{oldthebibliography}{#1}
    \setlength{\itemsep}{3ex}
  }{
    \end{oldthebibliography}
  }

\bibliographystyle{abbrv}
\bibliography{biblio}




\end{document}